\documentclass[11pt]{article}
\usepackage[english]{babel}
\usepackage[a4paper,left=3cm, right=3cm,top=3cm, bottom=4cm]{geometry}
\usepackage{amsmath,amssymb,amsfonts, marvosym, mathrsfs}
\usepackage{epsfig}
\usepackage{tikz}
\usepackage{float}
\usepackage{graphics}
\usepackage{bbm}
\usepackage{bm, bbm}
\usepackage{boxedminipage,verbatim,float,caption}
\usepackage{algorithm}
\usepackage{fancybox}
\usepackage{algpseudocode}
\usepackage{todonotes}
\usepackage{complexity}
\usepackage{IEEEtrantools}
\usepackage[thmmarks]{ntheorem}
\usepackage{soul}
\usepackage[]{hyperref}
 \hypersetup{colorlinks=true, citecolor=black,%
                 filecolor=black,%
                 linkcolor=black,%
                 urlcolor=black}

\usetikzlibrary[graphs, arrows, backgrounds, intersections, positioning, fit, petri, calc, shapes, decorations.pathmorphing]
\usetikzlibrary{arrows.meta}
\usepgflibrary{patterns}
\tikzset{node distance=1cm, bend angle=20,
vertex/.style={circle,minimum size=2mm,very thick, draw=black, fill=black, inner sep=0mm}, information text/.style={inner sep=1ex, font=\Large}, help lines/.style={-,color=black, >=stealth', shorten <=.5pt, shorten >=.5pt}, blue help lines/.style={help lines,color=darkblue}, red help lines/.style={help lines, color=darkred},
>={[scale=1.1]Stealth}}

\definecolor{mLightBrown}{HTML}{EB811B}

\setstcolor{mLightBrown}

\let\oldbibliography\thebibliography
\renewcommand{\thebibliography}[1]{%
  \oldbibliography{#1}%
  \setlength{\itemsep}{2pt}%
}

\makeatletter
\newtheoremstyle{prime}%
 {\item[\hskip\labelsep \theorem@headerfont ##1\ \theorem@separator]}%
{\item[\hskip\labelsep \theorem@headerfont ##1\ ##3' \theorem@separator]}
\makeatother

\makeatletter
\newtheoremstyle{proofof}
{\item[\hskip\labelsep \theorem@headerfont ##1\ \theorem@separator]}%
{\item[\hskip\labelsep \theorem@headerfont ##1\ ##3\theorem@separator]}
\makeatother

\makeatletter
\newtheoremstyle{restated}
  {\item[\hskip\labelsep \theorem@headerfont ##1\ ##2\theorem@separator]}%
  {\item[\hskip\labelsep {\theorem@headerfont ##1\ ##2}{\normalfont\ (##3)}{\theorem@headerfont
  \theorem@separator}]}
\makeatother

\makeatletter
\newtheoremstyle{nonumberrestated}%
  {\item[\theorem@headerfont \hskip\labelsep ##1\theorem@separator]}%
  {\item[\hskip\labelsep \theorem@headerfont ##3\theorem@separator]}%
\makeatother

\theoremsymbol{$\diamond$}
\newtheorem{theorem}{Theorem}
\newtheorem{lemma}[theorem]{Lemma}

\newtheorem{proposition}[theorem]{Proposition}

\newtheorem{claim}{Claim}

\theoremstyle{prime}

\theoremstyle{restated}
\newtheorem*{restatedtheorem*}{}

{\end{restatedtheorem*}}
\makeatother

\def \QD1 {\hfill $\spadesuit$}

\newcommand{\overr}[1]{\overset{\text{\tiny$\bm\rightarrow$}}{#1}}

\numberwithin{equation}{section}

\theoremstyle {nonumberplain}
\theoremseparator{:}
\theorembodyfont{\normalfont}
\theoremsymbol{\rule{1ex}{1ex}}
\newtheorem{proof}{Proof}

\theoremstyle{proofof}
\newtheorem{proofof}{Proof of}

\theoremsymbol{$\square$}
\newtheorem{proof2}{Proof}

\newcommand{\jou}[6]{\textsc{#1,} \emph{#2,} #3, #4 (#5), pp. #6.}
\newcommand{\elecjou}[6]{\textsc{#1,} \emph{#2,} #3, #4 (#5), #6.}
\def \JCTB {J. Combin. Theory \, Ser.~B}

\begin{document}
\title{\bf Digraphs and variable degeneracy}

\author{J{\o}rgen Bang-Jensen\thanks{Research supported by the Independent Research Fond Denmark under grant number DFF 7014-00037B}\\
\small IMADA\\[-0.8ex]
\small University of Southern Denmark\\[-0.8ex] 
\small Campusvej 55, DK-5320 Odense M, Denmark\\
\small\tt \href{mailto:jbj@imada.sdu.dk}{jbj@imada.sdu.dk}\\
\and
Thomas Schweser\footnotemark[1] \qquad  Michael Stiebitz\footnotemark[1]\\
\small Institute of Mathematics\\[-0.8ex]
\small Technische Universit\"at Ilmenau\\[-0.8ex]
\small D-98684 Ilmenau, Germany\\
\small\tt \{\href{mailto:thomas.schweser@tu-ilmenau.de}{thomas.schweser}, \href{mailto:michael.stiebitz@tu-ilmenau.de}{michael.stiebitz}\}@tu-ilmenau.de}

\date{}
\maketitle

\begin{abstract}
Let $D$ be a digraph, let $p \geq 1$ be an integer, and let $f: V(D) \to \mathbb{N}_0^p$ be a vector function with $f=(f_1,f_2,\ldots,f_p)$. We say that $D$ has an $f$-partition if there is a partition $(D_1,D_2,\ldots,D_p)$ into induced subdigraphs of $D$ such that for all $i \in [1,p]$, the digraph $D_i$ is weakly $f_i$-degenerate, that is, in every non-empty subdigraph $D'$ of $D_i$ there is a vertex $v$ such that $\min\{d_{D'}^+(v), d_{D'}^-(v)\} < f_i(v)$. In this paper, we prove that the condition $f_1(v) + f_2(v) + \ldots + f_p(v) \geq \max \{d_D^+(v),d_D^-(v)\}$ for all $v \in V(D)$ is almost sufficient for the existence of an $f$-partition and give a full characterization of the bad pairs $(D,f)$. Moreover, we describe a polynomial time algorithm that (under the previous conditions) either verifies that $(D,f)$ is a bad pair or finds an $f$-partition. Among other applications, this leads to a generalization of Brooks' Theorem as well as the list-version of Brooks' Theorem for digraphs, where a coloring of digraph is a partition of the digraph into acyclic induced subdigraphs. We furthermore obtain a result bounding the $s$-degenerate chromatic number of a digraph in terms of the maximum of maximum in-degree and maximum out-degree.
\end{abstract}



\bigskip
\noindent Most of our terminology is defined as in \cite{BaGu08} and similar to the papers~\cite{BaBeSchStEJC,BaBeSchStJGT}  (see also Section~\ref{section_definitions}). Throughout this paper, for $0 \leq \ell \leq k$, let $[\ell,k]=\{i \in \mathbb{N}_0 ~|~ \ell \leq i \leq k\}$. 

\section{Introduction}
Even though most people would define a $k$-coloring of an (undirected) graph $G$ as a function that assigns colors from a color set of cardinality $k$ to the vertices of $G$ such that the same-colored vertices induce edgeless subgraphs of $G$, this is nothing else than a partition of $G$ into vertex disjoint induced subgraphs $(G_1,G_2,\ldots,G_k)$ such that $G_i$ is edgeless for all $i \in [1,k]$. Naturally, both ways how to regard a coloring have their benefits. In this paper, we examine a more general approach regarding the latter definition that will allow us to obtain various well-known coloring results. But first of all, we need to clarify what digraph coloring refers to. A \textbf{coloring} and $k$\textbf{-coloring} of a digraph $D$ is a function $\varphi:V(D) \to [1,k]$ such that each color class $\varphi^{-1}(i)$ induces an acyclic subdigraph of $D$, that is, a subdigraph that does not contain any directed cycle. The \textbf{dichromatic number} $\overr{\chi}(D)$ of a digraph $D$ is the least integer $k$ such that $D$ admits a $k$-coloring. This digraph coloring concept was originally introduced by Neumann-Lara~\cite{NeuLa82} in 1982; however, it took over twenty years until it was rediscovered by Mohar~\cite{Mo03} in 2003. Ever since, it has attracted much attention amongst graph theorists (see, e.g., \cite{AhBeKf08, ArOl10, BaBeSchStEJC, BaBeSchStJGT,Gol16,Ha11, HaMo11.2, HaMo11, HaMo12,HoKa14, Mo04, Mo10}). 
Although this coloring concept might not seem intuitive at first sight, there are various factors stressing why it is especially reasonable. First of all, the dichromatic number of a bidirected graph and the chromatic number of its underlying (undirected) graph coincide. Consequently, many theorems on digraph coloring are generalizations of theorems on coloring of undirected graphs. Moreover, it has been shown that plenty of well-known theorems in graph coloring indeed have digraph counterparts. For example, Harutyunyan and Mohar~\cite{HaMo12} proved that there exist digraphs $D$ of maximum total degree $\Delta$ and arbitrary large digirth such that $\overr \chi(D) \geq \frac{c \Delta}{\log \Delta}$ for some constant $c$, thereby obtaining the analogue of a famous result of Bollob\'as~\cite{Bol78}, respectively Kostochka and Mazurova~\cite{KosMaz77}. Moreover, Andres and Hochst\"attler~\cite{AnHo15} obtained the digraph analogue of Chudnovsky, Robertson, Seymour, and Thomas' celebrated Strong Perfect Graph Theorem~\cite{CRST06}. In this paper, we will mainly focus on the analogue to Brooks' famous theorem~\cite{Brooks}, which was discovered by \textsc{Mohar} in 2010~\cite{Mo10}. Note that, given a digraph $D$, its maximum out-degree (respectively in-degree) is denoted by $\Delta^+(D)$ (respectively $\Delta^-(D)$).

\begin{theorem}[\textsc{Mohar}, 2010] \label{theorem_mohar-brooks}
Let $D$ be a connected digraph. Then, $D$ satisfies $\overr{\chi}(D) \leq \max \{\Delta^-(D), \Delta^+(D)\} + 1$ and equality holds if and only if $D$ is
\begin{itemize}
\item[\upshape (a)] a directed cycle of length $\geq 2$, or
\item[\upshape (b)] a bidirected cycle of odd length $\geq 3$, or
\item[\upshape (c)] a bidirected complete graph.
\end{itemize}
\end{theorem} 

In fact, it turns out that it is possible to obtain a choosability version of Brooks' Theorem for digraphs, \emph{i.e.}, a version regarding list-colorings. Given a digraph $D$, a color set $\Gamma$, and a function $L: V(D) \to 2^\Gamma$ (a so-called \textbf{list-assignment}), an $L$\textbf{-coloring} of $D$ is a function $\varphi:V(D) \to \Gamma$ such that $\varphi(v) \in L(v)$ for all $v \in V(D)$ and  $D[\varphi^{-1}(\alpha)]$ contains no directed cycle for each $\alpha \in \Gamma$ (if such a coloring exists, we say that $D$ is  $L$\textbf{-colorable}). Harutyunyan and Mohar \cite{HaMo11} proved the following, thereby extending a well-known theorem of Erd\H{o}s, Rubin and Taylor \cite{ErdRubTay79} for undirected graphs. 
Note that a \textbf{block} $B$ of a digraph is a maximal connected subdigraph that does not contain a separating vertex, \emph{i.e.}, the underlying graph $G(B)$ of $B$ is a block of the underlying graph $G(D)$ of $D$. By $\mathscr{B}(D)$ we denote the set of blocks of a digraph $D$. For $v \in V(D)$, $\mathscr{B}_v(D)$ denotes the set of blocks of $D$ containing $v$.

\begin{theorem}[\textsc{Harutyunyan} and \textsc{Mohar}, 2011]
\label{theorem_harut-mohar-list}
Let $D$ be a connected digraph, and let $L$ be a list-assignment such that $|L(v)| \geq \max \{d_D^+(v), d_D^-(v)\}$ for all $v \in V(D)$. Suppose that $D$ is not $L$-colorable. Then, the following statements hold:
\begin{itemize}
\item[\upshape (a)] $D$ is Eulerian and $|L(v)|= d_D^+(v) = d_D^-(v)$ for all $v \in V(D)$.
\item[\upshape (b)] If $B \in \mathscr{B}(D)$, then $B$ is {either a directed cycle of length $\geq 2$, a bidirected complete graph, or } a bidirected cycle of odd length $\geq 5$.
\item[\upshape (c)] For each $B \in \mathscr{B}(D)$ there is a set $\Gamma_B$ of $\Delta^+(B)$ colors such that for every $v \in V(D)$, the sets $\Gamma_B$ with $B \in \mathscr{B}_v(D)$ are pairwise disjoint and $L(v)= \bigcup_{B \in \mathscr{B}_v(D)} \Gamma_B$.
\end{itemize}
\end{theorem}

In the present paper, we will obtain a generalization of the two previously mentioned results by examining degenerate digraphs. The concept of digraph degeneracy was introduced by Bokal et al.~\cite{Bokal04} in 2004. Given a positive integer $k$, a digraph $D$ is \textbf{weakly $k$-degenerate} if every non-empty subdigraph $D'$ contains a vertex $v$ with $\min\{d_{D'}^+(v), d_{D'}^-(v)\} < k$. As a consequence, a digraph is acyclic if and only if it is weakly $1$-degenerate and so a coloring of a digraph coincides with a partition of the digraph into induced subdigraphs which all are weakly $1$-degenerate. We shall extend this definition to the case of variable degeneracy, based on the model of Borodin, Kostochka, and Toft~\cite{BorKosToft} for undirected graphs. Let $D$ be a digraph and let $h:V(D) \to \mathbb{N}_0$ be a function. Then, $D$ is \textbf{weakly $h$-degenerate} if every non-empty subdigraph $D'$ contains a vertex $v$ with  $\min\{d_{D'}^+(v), d_{D'}^-(v)\} < h(v)$. 
Clearly, if $h \equiv k$ is the constant function, then $D$ is weakly $h$-degenerate if and only if $D$ is weakly $k$-degenerate. 

We will connect the concept of degeneracy with partitions of digraphs. A \textbf{partition} and $p$-\textbf{partition} of a digraph $D$ is a sequence $(D_1,D_2,\ldots,D_p)$ of pairwise disjoint induced subdigraphs of $D$ such that $V(D)=V(D_1)\cup V(D_2) \cup \ldots \cup V(D_p)$. Now let $f:V(D) \to \mathbb{N}_0^p$ be a vector function. Then we denote by $f_i$ the $i$-th coordinate of $f$, \emph{i.e.} $f=(f_1,f_2,\ldots,f_p)$. Then, an \textbf{$f$-partition} of $D$ is a $p$-partition $(D_1,D_2,\ldots,D_p)$ of $D$ such that $D_i$ is weakly $f_i$-degenerate for $i \in [1,p]$. If $D$ admits an $f$-partition, then we also say that $D$ is \textbf{$f$-partitionable}. The main aim of this paper is to determine under which degree conditions $D$ {admits} an $f$-partition.
First of all, let us motivate why this is worthwhile considering. To this end, let $D$ be a digraph and let $L$ be a list-assignment for $D$. Moreover, let $\Gamma=\bigcup_{v \in V(D)} L(D)$ be the set of all colors appearing in some list. By renaming the colors if necessary we may assume $\Gamma=[1,p]$. Let $f: V(D) \to \mathbb{N}_0^p$ be the vector function with
$$f_i(v)=
\begin{cases}
1 \text{ if } i \in L(v), \text{ and}\\
0 \text{ if } i \not \in L(v)
\end{cases}$$ for $i \in [1,p]$ (see also Figure~\ref{fig_transforming_lists}).

\begin{figure}[H]
\centering
\resizebox{0.8\linewidth}{!}{
\begin{tikzpicture}
[node distance=1cm, bend angle=20,
vertex/.style={circle,minimum size=2mm,very thick, draw=black, fill=black, inner sep=0mm}, information text/.style={fill=red!10,inner sep=1ex, font=\Large},>={[scale=1.1]Stealth}]

\begin{scope}[xshift=-5cm]

\node[draw=none,minimum size=3cm,regular polygon,regular polygon sides=4] (a) {};
\foreach \x in {1,2,3,4}
\node[vertex] (v\x) at (a.corner \x) {};

\node[draw=none,minimum size=8cm,regular polygon,regular polygon sides=4] (b) {};
\foreach \x in {1,2,3,4}
\node[vertex] (u\x) at (b.corner \x) {};

\node at (v1) [label={east:$\{1,2,4\}$}] {};
\node at (v2) [label={west:$\{1,3\}$}] {};
\node at (v3) [label={west:$\{2,4\}$}] {};
\node at (v4) [label={east:$\{3,4\}$}] {};

\node at (u1) [label={45:$\{1,3,4\}$}] {};
\node at (u2) [label={135:$\{2,3\}$}] {};
\node at (u3) [label={225:$\{1,4\}$}] {};
\node at (u4) [label={315:$\{2,3,4\}$}] {};

\path[->]
(v1) edge (v2)
(v2) edge (v3)
(v3) edge (v4)
(v4) edge (v1)
(u1) edge (u4)
(u4) edge  node [name=l1, midway, sloped, below=2pt] {} (u3)
(u3) edge (u2)
(u2) edge (u1);

\path[<->]
(u1) edge (v1)
(u2) edge (v2)
(u3) edge (v3)
(u4) edge (v4);		
\node at (l1) [yshift=-1cm] {$(D,L)$};

\node (h1) at (u1) [xshift=1.8cm, yshift=1.8cm] {};
\node (h2) at (u2) [xshift=-1.8cm, yshift=1.8cm] {};
\node (h3) at (u3) [xshift=-1.8cm, yshift=-1.8cm] {};
\node (h4) at (u1) [xshift=1.8cm, yshift=-1.8cm] {};

\node (l1) [draw=black, rectangle, fit=(h1)(h2)(h3)(h4)] {};
\node (r1) at (l1.east) {};

\end{scope}

\begin{scope}[xshift=8cm]
\node[draw=none,minimum size=3cm,regular polygon,regular polygon sides=4] (a) {};
\foreach \x in {1,2,3,4}
\node[vertex] (v\x) at (a.corner \x) {};

\node[draw=none,minimum size=8cm,regular polygon,regular polygon sides=4] (b) {};
\foreach \x in {1,2,3,4}
\node[vertex] (u\x) at (b.corner \x) {};

\node at (v1) [label={east:$(1,1,0,1)$}] {};
\node at (v2) [label={west:$(1,0,1,0)$}] {};
\node at (v3) [label={west:$(0,1,0,1)$}] {};
\node at (v4) [label={east:$(0,0,1,1)$}] {};

\node at (u1) [label={45:$(1,0,1,1)$}] {};
\node at (u2) [label={135:$(0,1,1,0)$}] {};
\node at (u3) [label={225:$(1,0,0,1)$}] {};
\node at (u4) [label={315:$(0,1,1,1)$}] {};

\path[->]
(v1) edge (v2)
(v2) edge (v3)
(v3) edge (v4)
(v4) edge (v1)
(u1) edge (u4)
(u4) edge  node [name=l1, midway, sloped, below=2pt] {} (u3)
(u3) edge (u2)
(u2) edge (u1);

\path[<->]
(u1) edge (v1)
(u2) edge (v2)
(u3) edge (v3)
(u4) edge (v4);

\node at (l1) [yshift=-1cm] {$(D,f)$};

\node (h1) at (u1) [xshift=1.8cm, yshift=1.8cm] {};
\node (h2) at (u2) [xshift=-1.8cm, yshift=1.8cm] {};
\node (h3) at (u3) [xshift=-1.8cm, yshift=-1.8cm] {};
\node (h4) at (u1) [xshift=1.8cm, yshift=-1.8cm] {};

\node (l2) [draw=black, rectangle, fit=(h1)(h2)(h3)(h4)] {};
\node (r2) at (l2.west) {};

\end{scope}

\draw [shorten >=.5mm,-to,thick,decorate,
decoration={snake,amplitude=.4mm,segment length=2mm,
pre=moveto,pre length=1mm,post length=2mm}]
(r1) -- (r2); 

%
%

%
%
%
%
%
\end{tikzpicture}
}
\caption{Transforming a list-assignment into a vector function}
\label{fig_transforming_lists}
\end{figure}
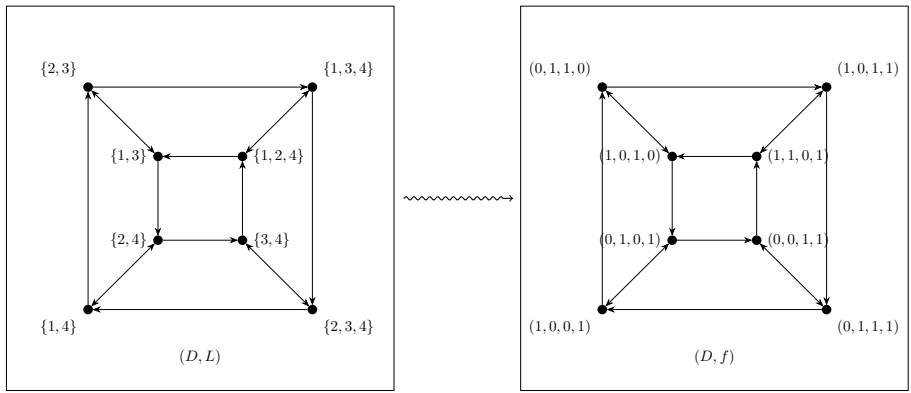

If $D$ is $L$-colorable, then it easy to confirm that the partition $(D_1,D_2,\ldots,D_p)$ of $D$ where $D_i$ is the subdigraph of $D$ induced by vertices of color $i$ is indeed an $f$-partition of $D$. Conversely, if $D$ has an $f$-partition $(D_1,D_2,\ldots,D_p)$, then setting $\varphi(v)=i$ if $v \in V(D_i)$ leads to an $L$-coloring of $D$. Hence, $D$ is $L$-colorable if and only if $D$ has an $f$-partition and so the task of finding an $f$-partition generalizes the usual list-coloring problem (and, of course, the usual coloring problem, too).

Since {deciding whether} the dichromatic number is at most two is already \NP-hard (see \cite{Bokal04}), it is pointless to try to determine whether a digraph $D$ is $f$-partitionable in general. Instead, we need to find a reasonable {condition} for the function $f$ {that would allow a digraph satisfying this condition to be $f$-partitionable}. Note that in the above transformation, it follows from the definition of $f$ that $$|L(v)| = f_1(v) + f_2(v) + \ldots + f_p(v)$$ for all $v \in V(D)$. Thus, Theorem~\ref{theorem_harut-mohar-list} suggests that $$f_1(v) + f_2(v) + \ldots + f_p(v) \geq \max \{d_D^+(v),d_D^-(v)\}$$ for all $v \in V(D)$ might be the right condition to investigate. Indeed, we will prove that this condition is always sufficient for the existence of an $f$-partition, unless $(D,f)$ belongs to the following, recursively defined class of configurations. Clearly, $D$ admits an $f$-partition if and only if each {connected} component of $D$ has one and, hence, it suffices to examine connected digraphs.

Let $D$ be a connected digraph, let $p \geq 1$, and let $f:V(D) \to \mathbb{N}_0^p$ be a vector function. We say that $(D,f)$ is a \textbf{hard pair} and that $D$ is {{\bf $f$-hard}} if one of the following four conditions hold.

\begin{itemize}
\item[(H1)] $D$ is a block, $D$ is Eulerian, and there exists an index $j \in [1,p]$ such that
$$f_i(v)=
\begin{cases}
d_D^+(v)=d_D^-(v) & \text{if } i=j, and\\
0 & \text{otherwise}
\end{cases}$$
for all $i \in [1,p]$ and for each $v \in V(D)$. In this case, we say that $(D,f)$ is a hard pair of type \textbf{(M)}.
\item[(H2)] $D$ is a bidirected complete graph and there are integers $n_1,n_2,\ldots,n_p \geq 0$ with at least two $n_i$ different from zero such that $n_1 + n_2 + \ldots + n_p=|D|-1$ and that
$$f(v)=(n_1,n_2,\ldots,n_p) \quad \text{for all } v \in V(D).$$
In this case, we say that $(D,f)$ is a hard pair of type \textbf{(K)}.
\item[(H3)] $D$ is a bidirected cycle of odd length and there are two indices $k \neq \ell$ from the set $[1,p]$ such that
$$f_i(v)=
\begin{cases}
1 & \text{if } i \in \{k,\ell\}, and\\
0 & \text{otherwise}
\end{cases}
$$ for all $i \in [1,p]$ and for each $v \in V(D)$. In this case, we say that $(D,f)$ is a hard pair of type \textbf{(C)}.
\item[(H4)] There are two hard pairs $(D^1,f^1)$ and $(D^2,f^2)$ with $f^j:V(D^j) \to \mathbb{N}_0^p$ for $j \in \{1,2\}$ such that $D$ is obtained from $D^1$ and $D^2$ by identifying two vertices $v^1 \in V(D^1)$ and $v^ 2 \in V(D^2)$ to a new vertex $v$. Furthermore, for $w \in V(D)$, it holds {that}
$$f(w)=
\begin{cases}
f^1(w) & \text{if } w \in V(D^1) \setminus \{v^1\}, \\
f^2(w) & \text{if } w \in V(D^2) \setminus \{v^2\}, \\
f^1(v^1) + f^2(v^2) & \text{if } w=v.
\end{cases}$$
In this case we say that $(D,f)$ is obtained from $(D^1,f^1)$ and $(D^2,f^2)$ by \textbf{merging} $v^1$ and $v^2$ to $v$.
\end{itemize}
In order to develop a better feeling of how hard pairs may look like, we refer the reader to Figure~\ref{fig_hard_pairs_example}. The main result of this paper is the following.

\begin{theorem}\label{theorem_main-result}
Let $D$ be a connected digraph, let $p \geq 1$ be an integer, and let $f:V(D) \to \mathbb{N}_0^p$ be a vector function such that $f_1(v) + f_2(v) + \ldots + f_p(v) \geq \max \{d_{D}^+(v), d_D^-(v)\}$ for all $v \in V(D)$. Then, $D$ is not $f$-partitionable if and only if $(D,f)$ is a hard pair.
\end{theorem}

\begin{figure}[htbp]
\centering
\resizebox{\linewidth}{!}{
\begin{tikzpicture} [node distance=1cm, bend angle=20, >={[scale=1.1]Stealth}]

\node[draw=none, minimum size=4cm, regular polygon, regular polygon sides=5, xshift=16cm](e){};
\node[vertex] (v51) [label={north:$(1,1,0)$}] at (e.corner 1) {};
\node[vertex] (v52) [label={north west:$(1,1,0)$}]at (e.corner 2) {};
\node[vertex] (v53) [label={south west:$(1,1,0)$}]at (e.corner 3) {};
\node[vertex] (v54) [label={south east:$(1,1,0)$}]at (e.corner 4) {};
\node[vertex] (v55) [label={north east:$(1,1,0)$}]at (e.corner 5) {};

\node[draw=none, minimum size=4cm, regular polygon, regular polygon sides=4, xshift=8cm] (d) {};

\node[vertex] (v41) [label={north east:$(2,0,1)$}] at (d.corner 1) {};
\node[vertex] (v42) [label={north west:$(2,0,1)$}] at (d.corner 2) {};
\node[vertex] (v43) [label={south west:$(2,0,1)$}] at (d.corner 3) {};
\node[vertex] (v44) [label={south east:$(2,0,1)$}] at (d.corner 4) {};

\node[draw=none, minimum size=4cm, regular polygon, regular polygon sides=6] (a) {};

\node[vertex] (v11) [label={north east:$(0,2,0)$}] at (a.corner 1) {};
\node[vertex] (v12) [label={north west:$(0,2,0)$}] at (a.corner 2) {};
\node[vertex] (v13) [label={west:$(0,1,0)$}] at (a.corner 3) {};
\node[vertex] (v14) [label={south west:$(0,2,0)$}] at (a.corner 4) {};
\node[vertex] (v15) [label={south east:$(0,2,0)$}] at (a.corner 5) {};
\node[vertex] (v16) [label={east:$(0,1,0)$}] at (a.corner 6) {};
\node[vertex] (v17) [label={east:$(0,4,0)$}]{};

\path[->]
		 (v11) edge (v12) 
		 (v12) edge (v13) 
		 (v13) edge (v14)
		 (v14) edge (v15)
		 (v15) edge (v16)
		 (v16) edge (v11);
\path[<->]
		(v17) edge (v11)
		(v17) edge (v12)
		(v17) edge (v14)
		(v17) edge (v15);

\path[<->]
		(v41)  edge  (v42)
			   edge  (v43)	   
			   edge  (v44)    
		 (v42) edge  (v43)	   
		 	   edge  (v44) 
		 (v43) edge  (v44);
		 	   
\path[<->]
		 (v51) edge (v52) 
		 (v52) edge (v53) 
		 (v53) edge (v54)
		 (v54) edge (v55) 
		 (v55) edge (v51);
		 
\begin{pgfonlayer}{background}
\node (w51) at (v51) [yshift=1cm]{};
\node (w52) at (v52) [xshift=-1.2cm]{};
\node (w53) at (v53) [xshift=-1.2cm, yshift=-1cm]{};
\node (w54) at (v54) [xshift=1.2cm, yshift=-1cm]{};
\node (w55) at (v55) [xshift=1.2cm]{};

\node [rectangle,  fit=(w51)(w52)(w53)(w54)(w55), label={[font=\Large, yshift=.7 cm]below:(C)}]{};

\begin{scope}
\node [rectangle, fit=(w51)(w52)(w53)(w54)(w55), xshift=-8cm, label={[font=\Large, yshift=.7 cm]below:(K)}]{};
\end{scope}

\begin{scope}
\node [rectangle, fit=(w51)(w52)(w53)(w54)(w55), xshift=-16cm, label={[font=\Large, yshift=.7 cm]below:(M)}]{};
\end{scope}
\end{pgfonlayer}		
		 
\end{tikzpicture}
}
\caption{Examples of hard pairs}
\label{fig_hard_pairs_example}
\end{figure}
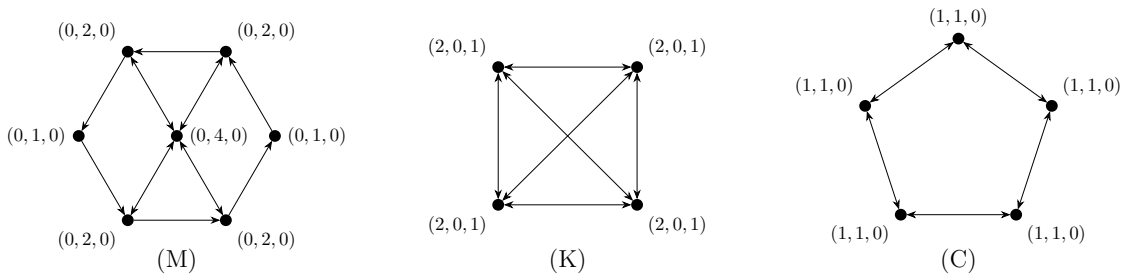

\section{Basic Terminology}\label{section_definitions}
Let $D=(V(D),A(D))$ be a digraph, where $V(D)$ is the \textbf{set of vertices} of $D$ and $A(D)$ is the \textbf{set of arcs} of $D$. The \textbf{order} $|D|$ of $D$ is the size of $V(D)$. The digraphs in this paper do not have loops nor parallel arcs; however, there may be two arcs in opposite directions between two vertices (in this case we say that the arcs are \textbf{opposite}). We denote by $uv$ the arc whose \textbf{initial vertex} is $u$ and whose \textbf{terminal vertex} is $v$. Two vertices $u,v$ are \textbf{adjacent} if at least one of $uv$ and $vu$ belongs to $A(D)$. If $u$ and $v$ are adjacent, we also say that $u$ is a \textbf{neighbor} of $v$ and vice versa. If $uv \in A(D)$, then $v$ is called an \textbf{out-neighbor} of $u$ and $u$ is called an \textbf{in-neighbor} of $v$. Given a digraph $D$ and a vertex set $X$,  we denote  {by $D[X]$ the subdigraph of $D$}  \textbf{induced} by the vertex set $X$, that is, $V(D[X])=X$ and $A(D[X])=\{uv \in A(D) ~|~ u,v \in X\}$. A digraph $D'$ is said to be an induced subdigraph of $D$ if $D'=D[V(D')]$. As usual, if $X$ is a subset of $V(D)$, we define $D-X=D[V(D) \setminus X]$. If $X=\{v\}$ is a singleton, we use $D-v$ rather than $D- \{v\}$. The \textbf{out-degree} of a vertex $v \in V(D)$, {denoted $d_D^+(v)$}, is the number of arcs whose inital vertex is $v$. Similarly, {the \textbf{in-degree} of $v$, denoted $d_D^-(v)$ is} number of arcs whose terminal vertex is $v$. A vertex $v \in V(D)$ is \textbf{Eulerian} if $d_D^+(v)=d_D^-(v)$. Moreover, the digraph $D$ is \textbf{Eulerian} if every vertex of $D$ is Eulerian.

The \textbf{underlying} graph $G(D)$ of $D$ is the simple undirected graph with $V(G(D))=V(D)$ and $\{u,v\}\in E(G(D))$ if and only if at least one of $uv$ and $vu$ belongs to $A(D)$. {A  {\bf component} of a digraph $D$ is a connected component of $G(D)$ and $D$ is {\bf connected} if it has precisely one component}. A \textbf{separating vertex} of a connected digraph $D$ is a vertex $v \in V(D)$ such that $D-v$ is not connected. Furthermore, a \textbf{block} of $D$ is a maximal subdigraph $D'$ of $D$ such that $D'$ has no separating vertex. By $\mathscr{B}(D)$, we denote the \textbf{set of blocks} of $D$. Moreover, if $v \in V(D)$, we denote by $\mathscr{B}_v(D)$ the set of blocks of $D$ containing $v$. A \textbf{bidirected} graph is a digraph that can be obtained from an undirected  (simple) graph $G$ by replacing each edge by two opposite arcs, we denote it by $D(G)$. 

\section{Proof of the main result}
In order to prove Theorem~\ref{theorem_main-result}, we need two classic results for undirected graphs. The first one is an easy consequence of Menger's Theorem and usually referred to as the Fan Lemma (see, e.g. \cite[Corollary 3.3.4]{Diestel}). The second lemma is due to Gallai~\cite[Satz 1.9]{Gal63a}. Recall that a \textbf{chord} of a cycle $C$ in an undirected graph $G$ is an edge of $G$ between vertices of $C$ that does not belong to the edges of $C$. 

\begin{lemma}[The Fan Lemma]\label{lemma_fan-lemma}
Let $G$ be a $k$-connected graph, let $v \in V(G)$, and let $X \subseteq V(G) \setminus \{v\}$ be a set of cardinality at least $k$. Then, there are $k$ paths from $v$  to vertices of $X$ whose only common vertex is $v$ and whose only intersection with $X$ are the respective end-vertices.
\end{lemma}

\begin{lemma}[Gallai, 1963]\label{lemma_gallai-chords}
If $G$ is a graph in which each even cycle has at least two chords, then every block of $G$ is a complete graph or an odd cycle. 
\end{lemma}

The proof of Theorem~\ref{theorem_main-result} is divided into two parts. In the first part, we obtain some properties of hard pairs and prove that hard pairs are not $f$-partitionable. The proof of the next proposition follows easily from the definition of hard pair and can be done via induction on the number of blocks of $D$.

\begin{proposition}\label{prop_digraph_block-hardpair}
Let $D$ be a connected digraph, let $p \geq 1$, and let $f :V(D) \to \mathbb{N}_0^p$ be a vector function such that $D$ is $f$-hard.  Then, for each $B \in \mathscr{B}(D)$ there is a uniquely determined function $f_B:V(B) \to \mathbb{N}_0^p$ such that the following statements hold:
\begin{itemize}
\item[\upshape (a)] $(B,f_B)$ is a hard pair of type {\upshape (M), (K),} or {\upshape(C)}.
\item[\upshape (b)] $f(v) = \sum_{B \in \mathscr{B}_v(H)}f_B(v)$ for all $v \in V(D)$. In particular, $f_B(v)=f(v)$ for all non-separating vertices $v$ of $D$ belonging to $B$.
\end{itemize}
\end{proposition}

\begin{proposition} \label{prop_digraph_f-hard}
Let $D$ be a connected digraph, and let $f :V(D) \to \mathbb{N}_0^p$ be a vector function with $p \geq 1$. If $(D,f)$ is a hard pair, then the following statements hold:
\begin{itemize}
\item[\upshape(a)] $f_1(v) + f_2(v) + \ldots + f_p(v) = d_D^+(v)=d_D^-(v)$ for all $v \in V(D)$. As a consequence, $D$ is Eulerian.
\item[\upshape(b)] $D$ is not $f$-partitionable.
\end{itemize}
\end{proposition}
\begin{proof}
Statement (a) follows from an easy induction on the number of blocks of $D$. Furthermore, if $D$ is a block, then $(D,f)$ is of type (M), (K), or (C) and it is an easy exercise to check that $D$ is indeed not $f$-partitionable. To show the reader how it may be done, suppose that $(D,f)$ is of type (K), \emph{i.e.}, $D=D(K_n)$ and there are integers $n_1,n_2,\ldots{}, n_p \geq 0$ with at least two $n_i$ different from zero such that $n_1 + n_2 + \ldots + n_p = n-1$ and that $f(v)=(n_1,n_2,\ldots,n_p)$ for all $v \in V(D)$. As $D$ is a bidirected complete graph, an induced subdigraph $D_i$ of $D$ is weakly $f_i$-degenerate if and only if $|D_i| \leq n_i$. Consequently, if there exists an $f$-partition $(D_1,D_2,\ldots,D_p)$, we have $$n = |D| = |D_1| + |D_2| + \ldots + |D_p| \leq n_1 + n_2 + \ldots + n_p = n-1,$$ which is impossible. This shows that $(D,f)$ is not of type (K).

 Now assume that $D$ is not a block {and} let $(D,f)$ be a minimal counter-example, \emph{i.e.}, $(D,f)$ is a hard pair, $D$ admits an $f$-partition, and $|D|$ is minimum with respect to the previous two conditions. As $D$ is not a block it follows with (H4) that there are two hard pairs $(D^1,f^1)$ and $(D^2,f^2)$ with $|D^j| < |D|$ for $j \in \{1,2\}$ such that $(D,f)$ is obtained from $(D^1,f^1)$ and $(D^2,f^2)$ by merging vertices $v^j \in V(D^j)$ to a new vertex $v$. For the sake of readability, we {use $v$ below for $v^1=v^2=v$}. By the choice of $(D,f)$, the digraph $D^j$ is not $f^j$-partitionable for $j \in \{1,2\}$. Now let $(D_1,D_2,\ldots,D_p)$ be an $f$-partition of $D$ and let $D^j_i=D^j \cap D_i$ for $j \in \{1,2\}$ and $i \in [1,p]$. By symmetry, we may assume that $v \in V(D_1)$. Then, $D^j_i$ is strictly $f^j_i$-degenerate for all $i \in [2,p]$ and $j \in \{1,2\}$ (as $D^j_i \subseteq D_i$ and $f^j_i(w)=f_i(w)$ for all $w \in V(D^j_i)$). {As $D^j$ is not $f^j$-partitionable, it follows that} $D^j_1$ is not $f^j_1$-partitionable for $j \in \{1,2\}$ and so there are non-empty subdigraphs $\tilde{D}^j \subseteq D^j_1$ with $\min\{d_{\tilde{D}^j}^+(w),d_{\tilde{D}^j}^-(w)\} \geq f^j_1(w)$ for all $w \in V(\tilde{D}^j)$. Let $\tilde{D}=\tilde{D^1} \cup \tilde{D^2}$. If $v \not \in \tilde{D}$, then $\tilde{D}$ is the disjoint union of $\tilde{D^1}$ and $\tilde{D^2}$ and we clearly have  $\min\{d_{\tilde{D}}^+(w),d_{\tilde{D}}^-(w)\} \geq f_1(w)$ for all $w \in V(\tilde{D})$. If $v \in \tilde{D}$, we obtain that $$f_1(v)=f^1_1(v) + f^2_1(v) \leq \min\{d_{\tilde{D}^1}^+(v),d_{\tilde{D}^1}^-(v)\} + \min\{d_{\tilde{D}^2}^+(v),d_{\tilde{D}^2}^-(v)\} \leq \min \{d_{\tilde{D}}^+(v), d_{\tilde{D}}^-(v) \}.$$
Consequently, $\tilde{D}$ is a subdigraph of $D_1$ with $\min\{d_{\tilde{D}}^+(w),d_{\tilde{D}}^-(w)\} \geq f_1(w)$ for all $w \in V(\tilde{D})$ and so $D_1$ is not weakly $f_1$-degenerate, {contradicting the assumption that $(D_1,D_2,\ldots{},D_p)$ is an $f$-partition of $D$.} 
\end{proof}

Thus, the ``if''-direction of Theorem~\ref{theorem_main-result} is proved. The hard part, \emph{i.e.}, the ``only if''-direction, is covered in the next theorem.

\begin{theorem} \label{theorem_non-partitionable_implies_f-hard}
Let $D$ be a connected digraph, let $p \geq 1$ be an integer, and let $f :V(D) \to \mathbb{N}_0^p$  be a vector function such that $f_1(v) + f_2(v) + \ldots + f_p(v) \geq \max \{d_D^+(v),d_D^-(v)\}$ for all $v \in V(D)$. If $D$ is not $f$-partitionable, then $(D,f)$ is a hard pair.
\end{theorem}

\begin{proof}
The proof is by reductio ad absurdum. So let $(D,f)$ be a smallest counterexample, that is,
\begin{itemize}
\item[\upshape (1)] $f_1(v) + f_2(v) + \ldots + f_p(v) \geq \max \{d_D^+(v),d_D^-(v)\}$ for all $v \in V(D)$,
\item[\upshape (2)] $D$ is not $f$-partitionable,
\item[\upshape (3)] $(D,f)$ is not a hard pair, and
\item[\upshape (4)] $|D|$ is minimum subject to (1),(2), and (3).
\end{itemize}
In order to derive a contradiction, we establish a sequence of eight claims.
\begin{claim}\label{claim_D-v}
{$D-v$ is $f$-partitionable for every $v\in V(D)$}.
\end{claim}
\begin{proof2}
{Suppose that $D-v$ is not $f$-partitionable for some $v\in V(D)$. Then} there is a component $D'$ of $D-v$ such that $D'$ is not $f$-partitionable and so $(D',f)$ is a hard pair (by (4)). As $D$ is connected, there is a neighbor $u$ of $v$ in $D'$. Since $(D',f)$ is a hard pair, Proposition~\ref{prop_digraph_f-hard}(a) implies that $f_1(u)+f_2(u)+ \ldots + f_p(u) =d_{D'}^+(u)=d_{D'}^-(u)$. As $D$ contains at least one of the arcs $vu$ and $uv$, we conclude that 
$$f_1(u)+f_2(u)+ \ldots + f_p(u) =d_{D'}^+(u)=d_{D'}^-(u) < \max \{d_D^+(u),d_D^-(u)\},$$
{contradicting (1)}.
\end{proof2}

The next claim is central for the proof of Theorem~\ref{theorem_non-partitionable_implies_f-hard}. Casually speaking, it says that not only is $D$ Eulerian and in- and out-degree of each vertex coincides with the sum of its $f$-values, but also, given a fixed vertex $v$ and an $f$-partition $(D_1,D_2,\ldots,D_p)$ of $D-v$, the component of each $D_i + v$ containing $v$ is Eulerian, too, and the respective degrees coincide with the $f_i$-values.  
 \begin{claim} \label{claim_non-partitionable_eulerian}
 Let $v \in V(D)$ be an arbitrary vertex, let $(D_1,D_2,\ldots,D_p)$ be an $f$-partition of $D-v$, and let $i \in [1,p]$. Then, the {following hold}:
 \begin{itemize}
 \item[\upshape (a)] $d_{D_i + v}^+(v) = d_{D_i+v}^-(v)=f_i(v)$.
 \item[\upshape (b)] $f_1(v) + f_2(v) + \ldots + f_p(v) = d_D^+(v) = d_D^-(v)$. As a consequence, $D$ is Eulerian.
 \item[\upshape (c)] Let $u$ be a neighbor of $v$ in $D_i$. Then, the sequence $(D_1',D_2', \ldots, D_p')$ with $D_i'=(D_i + v) - u$ and $D_j' = D_j$ for {$j \neq i$}  is an $f$-partition of $D-u$.
 \item[\upshape (d)] The component $D'$ of $D_i + v$ that contains $v$ is Eulerian and $d_{D'}^+(w) = d_{D'}^-(w)=f_i(w)$ for all $w \in V(D')$.
 \end{itemize}
 \end{claim}
 \begin{proof2}
Let $i \in [1,p]$ be arbitrary. As $D$ is not $f$-partitionable, the digraph $D_i + v$ is not weakly $f_i$-degenerate. Thus, there is a subdigraph $D''$ of $D_i + v$ such that $\min \{d_{D''}^+(w), d_{D''}^-(w)\} \geq f_i(w)$ for all $w \in V(D'')$. Clearly, $D''$ must contain $v$ (as $D_i$ is weakly $f_i$-degenerate) and so
 $$f_i(v) \leq \min \{d_{D''}^+(v), d_{D''}^-(v)\} \leq \min \{d_{D_i + v}^+(v), d_{D_i + v}^-(v)\}.$$
 Since $i$ was chosen arbitrarily, we conclude that
\begin{align*}
 \sum_{i\in [1,p]} f_i(v) & \leq \sum_{i\in [1,p]} \min \{d_{D_i + v}^+(v), d_{D_i + v}^-(v)\}  \leq \min \{d_D^+(v), d_D^-(v)\}\\
 & \leq \max \{d_D^+(v), d_D^-(v)\}  \leq  \sum_{i\in [1,p]} f_i(v)
\end{align*} 
and so we have equality everywhere. Thus, (a) and (b) {hold}. 

For the proof of (c) and (d), let $D'$ be the component of $D_i + v$ containing $v$, and let $u$ be a neighbor of $v$ in $D_i + v$ and therefore in $D'$ (if such a vertex does not exist, then $D'$ consists only of $v$, and by (a) $f_i(v) = 0=d_{D_i + v}^+(v) = d_{D_i + v}^- (v)$; so there is nothing to prove). Again, let $D''$ be the subdigraph of $D_i + v$ with $\min \{d_{D''}^+(w), d_{D''}^-(w)\} \geq f_i(w)$ for all $w \in V(D'')$. By (a),
$$f_i(v) = d_{D_i + v}^+(v) = d_{D_i + v}^- (v) \geq \max \{d_{D''}^+(v), d_{D''}^-(v)\}\geq f_i(v),$$
implying that the digraph $D''$ must contain all neighbors of $v$ in $D'$ and, in particular, $D''$ contains $u$. Consequently, $(D_i + v) - u$ is weakly $f_i$-degenerate (as $D''$ was chosen arbitrarily and so any ''bad`` subdigraph must contain $u$) and, hence, $(D_1',D_2',\ldots,D_p')$ with $D_i'=(D_i + v) - u$ and $D_j'=D_j$ for {$j \neq i$} is an $f$-partition of $D-u$, which proves (c). Clearly, the component of $D_i' + u$ containing $u$ is still $D'$ and so applying statement (a) to  the vertex $u$ and the partition $(D_1',D_2',\ldots,D_p')$ leads to $d_{D_i + v}^+ (u) =d_{D_i + v}^- (u) = f_i(u)$ (as $D_i + v = D_i' + u$). Note that if $z$ is a neighbor of $u$ in $D'$ we can swap $u$ and $z$, \emph{i.e.}, regard the $f$-partition $(D_1'', D_2'', \ldots, D_p'')$ with $D_i''=(D_i' + u) - z$ and $D_j''=D_j'=D_j$ for $j \neq i$, and obtain the same {conclusion} for $z$. By repeating this procedure, we eventually reach every vertex of $D'$ and so (d) follows.
\end{proof2}
 \begin{claim} \label{claim_block}
 $D$ is a block.
 \end{claim}
 \begin{proof2}
 Suppose, to the contrary, that $D$ is the union of two induced subdigraphs $D^1$ and $D^2$ with $V(D^1) \cap V(D^2) = \{v\}$ and $|D^j| < |D|$ for $j \in \{1,2\}$. We will model two functions $f^1$ and $f^2$ such that $(D^1,f^1)$ and $(D^2,f^2)$ are hard pairs and $(D,f)$ is obtained from the two hard pairs via the merging operation, thereby giving us the desired contradiction. To this end, let $(D_1,D_2,\ldots,D_p)$ be an $f$-partition of $D-v$ and, for $i \in [1,p]$ and $j \in \{1,2\}$, let $D_i^j=D_i \cap D^j$. Clearly, the digraphs $D_i^1$ and $D_i ^2$ are disjoint for all $i \in [1,p]$. Now let $i \in [1,p]$ and let $D'$ be the component of $D_i + v$ containing $v$. Then, $$D'=(D' \cap (D_i^1 + v)) \cup (D' \cap (D_i^2 + v)).$$ Note that $(D' \cap (D_i^1 + v))$ and $(D' \cap (D_i^2 + v))$ have only $v$ in common and that there are no arcs between vertices of $D_i^1$ and $D_i^2$. By Claim~\ref{claim_non-partitionable_eulerian}(d), we have $d_{D'}^+(w)=d_{D'}^-(w)=f_i(w)$ for all $w \in V(D')$. Thus, it follows from the above observation that 
$$f_i (w) = d_{D'}^+(w) = d_{D'}^-(w) = d_{(D' \cap (D_i^j + v))}^+(w)=d_{(D' \cap (D_i^j + v))}^-(w)$$ 
 for $j \in \{1,2\}$ and all $w \in V(D_i^j)$. Consequently, all vertices from $D' \cap (D_i^j + v))$ besides $v$ are Eulerian in $D' \cap (D_i^j + v))$. Since in each digraph, the sum of out-degrees over all vertices equals the sum of in-degrees over all vertices, we conclude that vertex $v$ is also Eulerian in $D' \cap (D_i^j + v)$ and, therefore, in $D_i^j + v$ for $j \in \{1,2\}$, \emph{i.e.} 
\begin{align}
d_{D_i^j + v}^+ (v) =d_{(D' \cap (D_i^j + v))}^+(v)=d_{(D' \cap (D_i^j + v))}^-(v)= d_{D_i^j + v}^- (v).
\label{align_d_D_i^j^+=d_D_i^j^-}
\end{align}

 This gives us a nice way to define the functions $f^1$ and $f^2$. For $i \in [1,p]$ and $j \in \{1,2\}$ let
 \begin{align}f_i^j(w)=
 \begin{cases}
 f_i(w) \quad \text{if } w \in V(D^j) \setminus \{v\}, \text{ and}\\
 d^+_{D_i^j}(v) \quad \text{if } w=v.
\end{cases}
\label{align_definition_f_i}
\end{align}
Then, for $j \in \{1,2\}$ and all $w \in V(D^j) \setminus \{v\}$ we have 
\begin{IEEEeqnarray}{rCl}
f_1^j(w) + f_2^j(w) \ldots + f_p^j(w)  & =  & f_1(w) + f_2(w) + \ldots + f_p(w) \nonumber \\
& = & \max\{d_D^+(w), d_D^-(w)\} = \max \{d_{D^j}^+(w), d_{D^j}^-(w)\}.
\label{align_f_values_block_claim}
\end{IEEEeqnarray}
Moreover, since $v$ is Eulerian in $D_i^j + v$ we obtain 
\begin{align}
d_{D^j + v}^+(v) = \sum_{i \in [1,p]} d_{D_i^j + v}^+ (v) =  \sum_{i \in [1,p]} d_{D_i^j + v}^- (v) =d_{D^j + v}^-(v),
\label{align_d_D^j+v}
\end{align}
and by the choice of $f^j(v)$, it follows that 
\begin{align}
f^j_1(v) + f^j_2(v) + \ldots + f^j_p(v) = d_{D^j } ^+ (v) = d_{D^j}^-(v).
\label{align_d_D^j^-v}
\end{align}
As a consequence, { both $(D^1,f^1)$ and $(D^2,f^2)$  fulfil} the requirements of {Theorem \ref{theorem_non-partitionable_implies_f-hard}}. 

First assume that for some $j \in \{1,2\}$, the digraph $D^j$ is $f^j$-partitionable. By symmetry, we may assume $j=1$. Then, $D^1$ admits an $f^1$-partition $(D_1',D_2',\ldots,D_p')$ and, by symmetry, $v \in V(D_1')$. Let $(D_1^*,D_2^*,\ldots,D_p^*)$ be a partition of $D$ with $D_1^* = D_1' \cup (D_1^2 + v)$ and $D_i^*= D_i' \cup D_i^2$ for $i \in [2,p]$. As $D_i'$ and $D_i^2$ are disjoint and as $(D_1^2,D_2^2, \ldots, D_p^2)$ is an $f$-partition of $D^2 - v$ , $D_i^*$ is weakly $f_i$-degenerate for $i \in [2,p]$. We claim that $D_1^*$ is weakly $f_1$-degenerate. To this end, let $\tilde{D}$ be a non-empty subdigraph of $D_1^*$. If $\tilde{D}$ is a subdigraph of $D_1^2$, then $\tilde{D}$ is weakly $f_1$-degenerate (as $D_1^2$ is weakly $f_1$-degenerate). So assume that $\tilde{D}^1 = \tilde{D} \cap D^1$ is non-empty. Since $\tilde{D}^1$ is a subdigraph of $D_1'$ and therefore weakly $f_1^1$-degenerate, there is a vertex $w \in V(\tilde{D}^1)$ with $\min \{d_{\tilde{D}^1}^+(w), d_{\tilde{D}^1}^-(w)\} < f_1^1(w)$. If $w \neq v$, then 
\begin{align}
\min \{d_{\tilde{D}}^+(w), d_{\tilde{D}}^-(w)\} = \min \{d_{\tilde{D}^1}^+(w), d_{\tilde{D}^1}^-(w)\} < f_1^1(w) = f_1(w),
\label{align_w=v}
\end{align}
 and we are done. It remains to consider the case that $w=v$. Since $d_{D_1^2 + v}^+(v) = d_{D_1^2 + v}^-(v) = f_1^2(v)$ (by \eqref{align_d_D_i^j^+=d_D_i^j^-} and \eqref{align_definition_f_i}), this implies that
$$\min\{d_{\tilde{D}}^+(v), d_{\tilde{D}}^-(v)\} \leq \min\{d_{\tilde{D}^1}^+(v), d_{\tilde{D}^1}^-(v)\} + d_{D_1^2}^+(v) < f_1^1(v) + f_1^2(v) = f(v).$$
Consequently, $D_1^*$ is weakly $f_1$-degenerate, as claimed, and so $(D_1^*,D_2^*,\ldots,D_p^*)$ is an $f$-partition of $D$, which is impossible. 

Thus, $D^j$ is not $f^j$-partitionable for $j \in \{1,2\}$ and, hence { by the minimality of $D$}, both $(D^1,f^2)$ and $(D^2,f^2)$ are hard pairs  and $(D,f)$ is obtained from $(D^1,f^2)$ and $(D^2,f^2)$ via the merging operation. As a consequence, $(D,f)$ is a hard pair, contradicting (3). This contradiction completes the proof of the claim.
\end{proof2}

By the above claim, $D$ is a block. Hence it remains to show that $(D,f)$ is a hard pair of type (M), (K), or (C), giving us the desired contradiction. The next claim eliminates the pairs of type (M).
\begin{claim} \label{claim_two-colors}
For every $v \in V(D)$ and each $f$-partition $(D_1,D_2,\ldots,D_p)$ of $D-v$, there are two indices $i \neq j$ from $[1,p]$ such that $D_i$ and $D_j$ are non-empty.
\end{claim}
\begin{proof2}
{Suppose} that there is a vertex $v \in V(D)$ and a partition $(D_1,D_2,\ldots,D_p)$ of $D-v$ such that exactly one part of the partition is non-empty, say $D_1$. {Then } $D_1 + v = D$ and so it follows from Claim~\ref{claim_non-partitionable_eulerian}(d) {and the fact that $D$ is connected} that $d_D^+(w)=d_D^-(w)=f_1(w)$ for all $w \in V(D)$. {Thus it follows from  Claim~\ref{claim_non-partitionable_eulerian}(b) that } $f_j(w)=0$ for all $j \neq i$ from $[1,p]$ and for all $w \in V(D)$  and $(D,f)$ is a hard pair of type (M), {contradicting (3)}.
\end{proof2}

Actually, Claim~\ref{claim_non-partitionable_eulerian}(c) provides us with a powerful tool that we shall use in the following. Let $v \in V(D)$ be an arbitrary vertex and let $(D_1,D_2,\ldots,D_p)$ be an $f$-partition of $D-v$. Moreover, let $u \in V(D)$ be a neighbor of $v$. Then, $u \in V(D_i)$ for some $i \in [1,p]$ and Claim~\ref{claim_non-partitionable_eulerian}(c) implies that replacing $D_i$ with $(D_i + v) - u$ leads to an $f$-partition of $D-u$ in which $v$ is contained in what was previously $D_i$. Thus, we can swap $v$ and any neighbor $u$ of $v$ and obtain a new $f$-partition of $D-u$. In order to make this observation a bit more graphical, we introduce the following terms. Given a vertex $v \in V(D)$, we call an $f$-partition $(D_1,D_2,\ldots,D_p)$ of $D-v$ an $f$-\textbf{coloring} $\varphi$ of {$D-v$}. Moreover, for $w \in V(D)\setminus \{v\}$, we set $\varphi(w)=i$ if $w \in V(D_i)$ and say that $w$ has \textbf{color} $i$. Finally, we say that the vertex $v$ is \textbf{uncolored}.
The following Claim is {just a}  reformulation of Claim~\ref{claim_non-partitionable_eulerian}(a) and Claim~\ref{claim_two-colors} to fit the new terminology.

\begin{claim}\label{claim_degree_in_D_i}
Let $v \in V(D)$ be an arbitrary vertex and let $\varphi$ be an $f$-coloring of $D-v$. Then, the following statements hold:
\begin{itemize}
\item[\upshape (a)] At least two color-classes of $\varphi$ are non-empty.
\item[\upshape (b)] $d_{D[\varphi^{-1}(\{i\}) ]+v}^+(v) = d_{D[\varphi^{-1}(\{i\})] +v}^-(v) =f_i(v) \text{ for all } i \in [1,p].$
\end{itemize}
\end{claim}

Using this terminology, the above described method of swapping two vertices is nothing else than assigning to $v$ the color of some neighbor $u$ of $v$ and uncoloring $u$. We will call this process \textbf{shifting} the color from $u$ to $v$. Note that this leads to an $f$-coloring of $D-u$. The original idea of shifting goes back to \textsc{Gallai}~\cite{Gal63a}. 

Now let $C$ be a cycle in $G(D)$, let $v \in V(C)$ be an arbitrary vertex, and let $\varphi$ be an $f$-coloring of $D-v$. Moreover, let $u$ and $w$ be the vertices such that $uv$ and $vw$ are in $E(C)$. Then, we can shift the color from $u$ to $v$. Afterwards, we shift the color from the other neighbor of $u$ on $C$ to $u$. Continuing like this, we can shift the color of each vertex of $C$, one after another, clockwise, until eventually we shift the color from $v$ to $w$. This gives us  a new $f$-coloring $\varphi'$ of $D-v$ (see Figure~\ref{fig_shifting-example}). In particular, $\varphi'(w)=\varphi(u)$. 

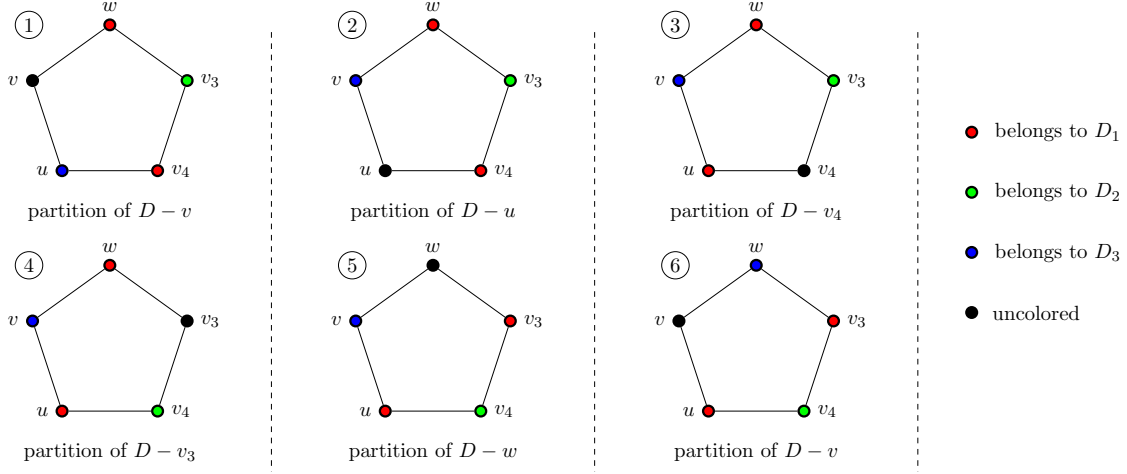
\begin{figure}[H]
\centering
\resizebox{\linewidth}{!}{
\begin{tikzpicture}[>={[scale=1.1]Stealth}]
\node[draw=none,minimum size=3cm,regular polygon,regular polygon sides=5] (a) {};

\node[vertex, fill=red, label={above:$w$}] (a1) at (a.corner 1) {};
\node[vertex, label={left:$v$}] (a2) at (a.corner 2){};
\node[vertex, fill=blue, label={left:$u$}] (a3) at (a.corner 3){};
\node[vertex, fill=red, label={right:$v_4$}] (a4) at (a.corner 4){};
\node[vertex, fill=green, label={right:$v_3$}] (a5) at (a.corner 5){};
\node[circle, inner sep=2pt, draw=black] at (a1) [xshift=-1.5cm]{$1$};
\node (l1) at (a1) [yshift=-3.5cm] {partition of $D-v$};

\path[-]
(a1) edge [-] (a2)
(a2) edge [-] (a3)
(a3) edge [-] (a4)
(a4) edge [-] (a5)
(a1) edge [-] (a5);

\node (h1) at (a1) [xshift=3cm]{};
\node (h2) at (h1) [yshift=-8.5cm]{};
\draw[dashed] (h1) -- (h2);

\begin{scope}[xshift=6cm]{};
\node[draw=none,minimum size=3cm,regular polygon,regular polygon sides=5] (a) {};

\node[vertex, fill=red, label={above:$w$}] (a1) at (a.corner 1) {};
\node[vertex, fill= blue, label={left:$v$}] (a2) at (a.corner 2){};
\node[vertex, label={left:$u$}] (a3) at (a.corner 3){};
\node[vertex, fill=red, label={right:$v_4$}] (a4) at (a.corner 4){};
\node[vertex, fill=green, label={right:$v_3$}] (a5) at (a.corner 5){};

\path[-]
(a1) edge [-] (a2)
(a2) edge [-] (a3)
(a3) edge [-] (a4)
(a4) edge [-] (a5)
(a1) edge [-] (a5);

\node (h1) at (a1) [xshift=3cm]{};
\node (h2) at (h1) [yshift=-8.5cm]{};
\draw[dashed] (h1) -- (h2);
\node (l1) at (a1) [yshift=-3.5cm] {partition of $D-u$};
\node[circle, inner sep=2pt, draw=black] at (a1) [xshift=-1.5cm]{$2$};
\end{scope}

\begin{scope}[xshift=12cm]{};
\node[draw=none,minimum size=3cm,regular polygon,regular polygon sides=5] (a) {};

\node[vertex, fill=red, label={above:$w$}] (a1) at (a.corner 1) {};
\node[vertex, fill= blue, label={left:$v$}] (a2) at (a.corner 2){};
\node[vertex, fill=red, label={left:$u$}] (a3) at (a.corner 3){};
\node[vertex, label={right:$v_4$}] (a4) at (a.corner 4){};
\node[vertex, fill=green, label={right:$v_3$}] (a5) at (a.corner 5){};

\path[-]
(a1) edge [-] (a2)
(a2) edge [-] (a3)
(a3) edge [-] (a4)
(a4) edge [-] (a5)
(a1) edge [-] (a5);

\node (l1) at (a1) [yshift=-3.5cm] {partition of $D-v_4$};
\node[circle, inner sep=2pt, draw=black] at (a1) [xshift=-1.5cm]{$3$};
\end{scope}

\begin{scope}[yshift=-4.5cm]{};
\node[draw=none,minimum size=3cm,regular polygon,regular polygon sides=5] (a) {};

\node[vertex, fill=red, label={above:$w$}] (a1) at (a.corner 1) {};
\node[vertex, fill= blue, label={left:$v$}] (a2) at (a.corner 2){};
\node[vertex, fill=red, label={left:$u$}] (a3) at (a.corner 3){};
\node[vertex, fill=green, label={right:$v_4$}] (a4) at (a.corner 4){};
\node[vertex, label={right:$v_3$}] (a5) at (a.corner 5){};

\path[-]
(a1) edge [-] (a2)
(a2) edge [-] (a3)
(a3) edge [-] (a4)
(a4) edge [-] (a5)
(a1) edge [-] (a5);

\node (l1) at (a1) [yshift=-3.5cm] {partition of $D-v_3$};
\node[circle, inner sep=2pt, draw=black] at (a1) [xshift=-1.5cm]{$4$};
\end{scope}

\begin{scope}[yshift=-4.5cm, xshift=6cm]{};
\node[draw=none,minimum size=3cm,regular polygon,regular polygon sides=5] (a) {};

\node[vertex, label={above:$w$}] (a1) at (a.corner 1) {};
\node[vertex, fill= blue, label={left:$v$}] (a2) at (a.corner 2){};
\node[vertex, fill=red, label={left:$u$}] (a3) at (a.corner 3){};
\node[vertex, fill=green, label={right:$v_4$}] (a4) at (a.corner 4){};
\node[vertex, fill=red, label={right:$v_3$}] (a5) at (a.corner 5){};

\path[-]
(a1) edge [-] (a2)
(a2) edge [-] (a3)
(a3) edge [-] (a4)
(a4) edge [-] (a5)
(a1) edge [-] (a5);

\node (l1) at (a1) [yshift=-3.5cm] {partition of $D-w$};
\node[circle, inner sep=2pt, draw=black] at (a1) [xshift=-1.5cm]{$5$};
\end{scope}

\begin{scope}[yshift=-4.5cm, xshift=12cm]{};
\node[draw=none,minimum size=3cm,regular polygon,regular polygon sides=5] (a) {};

\node[vertex, fill=blue, label={above:$w$}] (a1) at (a.corner 1) {};
\node[vertex, label={left:$v$}] (a2) at (a.corner 2){};
\node[vertex, fill=red, label={left:$u$}] (a3) at (a.corner 3){};
\node[vertex, fill=green, label={right:$v_4$}] (a4) at (a.corner 4){};
\node[vertex, fill=red, label={right:$v_3$}] (a5) at (a.corner 5){};

\path[-]
(a1) edge [-] (a2)
(a2) edge [-] (a3)
(a3) edge [-] (a4)
(a4) edge [-] (a5)
(a1) edge [-] (a5);

\node (l1) at (a1) [yshift=-3.5cm] {partition of $D-v$};
\node[circle, inner sep=2pt, draw=black] at (a1) [xshift=-1.5cm]{$6$};
\end{scope}

\begin{scope}[yshift=-2.25cm, xshift=16cm]
\node[vertex, yshift=.625cm, fill=green] (c2) {};
\node at (c2) [xshift=1.6cm] {belongs to $D_2$};
\node[vertex, fill=red] (c1) at (c2) [yshift=1.125cm]{};
\node at (c1) [xshift=1.6cm] {belongs to $D_1$};
\node[vertex, fill=blue] (c3) at (c2)[yshift=-1.125cm]{};
\node at (c3) [xshift=1.6cm] {belongs to $D_3$};
\node[vertex, fill=black] (c4) at (c3)[yshift=-1.125cm]{};
\node at (c4) [xshift=1.2cm] {uncolored};
\node (h3) at (h1) [xshift=6cm]{};
\node (h4) at (h2) [xshift=6cm]{};
\draw[dashed] (h3) -- (h4);
\end{scope}

\end{tikzpicture}
}
\caption{Clockwise shifting of colors around a cycle in $G(D)$.}
\label{fig_shifting-example}
\end{figure}

Similarly, starting from $\varphi$ with shifting the color from $w$ to $v$, we can shift the color of each vertex counter-clockwise on the cycle and obtain a third  $f$-coloring of $D-v$. By repeated clockwise, respectively counter-clockwise shifting, it is easy to see that the following claim is true.

\begin{claim} \label{claim_shifting_consecutive_vertices}
Let $C$ be a cycle in $G(D)$, let $v \in V(C)$ be an arbitrary vertex, and let $\varphi$ be an $f$-coloring of $D-v$. Moreover, let $u$ and $w$ be {the neighbours of $v$ on $C$, that is,} $\{uv,vw\} \subseteq E(C)$. Then, for any pair $v_1,v_2$ of vertices distinct from $v$ with $v_1v_2 \in E(C)$ and for each $i \in \{1,2\}$ there is an $f$-coloring $\varphi^*$ of $D-v$ such that
 $\varphi^*(u)=\varphi(v_i)$ and $\varphi^*(w)=\varphi(v_{3-i}).$
\end{claim}

{Using} Claims \ref{claim_degree_in_D_i} and~\ref{claim_shifting_consecutive_vertices}, we are able to prove the next claim.
\begin{claim}\label{claim_no-chords}
Let $C$ be a cycle in $G(D)$ and let $v \in V(C)$ be a vertex of $C$ that is not contained in a chord of $C$ in $G(D)$. Then, $C$ is an odd cycle and there is an $f$-coloring $\varphi^*$ of $D-v$ and indices $k \neq \ell$ from $[1,p]$ such that the vertices of $C-v$ are colored alternately with $k$ and $\ell$.
\end{claim}
\begin{proof2}
Let $u$ and $w$ be the vertices with $\{uv, vw\} \subseteq E(C)$. As $v$ is not contained in a chord, $u$ and $w$ are the only neighbors of $v$ from $V(C)$ in $G(D)$. We first claim that there is an $f$-coloring $\varphi^*$ of $D-v$ with $\varphi^*(u) \neq \varphi^*(w)$. To this end, let $\varphi$ be an $f$-coloring of $D-v$ and assume that $\varphi(u) = \varphi(w)$. From Claim~\ref{claim_degree_in_D_i}(a) we know that at least two color classes of $\varphi$ are non-empty. Thus, in $D-v$ there is a vertex $z$ with $\varphi(z) \neq \varphi(u)$. As $C$ is a cycle contained in the block $G(D)$, we have $|G(D)| \geq 3$ and so $G(D)$ is $2$-connected. Then, it follows from Lemma~\ref{lemma_fan-lemma} that in $G(D)$ there are two paths $P$, $P'$ from $z$ into the set $\{u,v,w\}$ whose only common vertex is $z$ and whose internal vertices are not in $\{u,v,w\}$. Note that the union $P \cup P'$ together with either one or both of the edges $\{uv, vw\}$ forms a cycle $C'$ in $G(D)$, which contains $v$. Since $\varphi(z) \neq \varphi(u)=\varphi(w)$, $C'$ contains two consecutive vertices of different colors, say $v_1$ and $v_2$. If $C'$ contains both $u$ and $w$, it follows from Claim~\ref{claim_shifting_consecutive_vertices} that there is a coloring $\varphi^*$ of $D-v$ with $\varphi^*(u)=\varphi(v_1)\neq \varphi(v_2)=\varphi^*(w)$, and we are done. In order to complete the first part of the claim, it remains to consider the case that $C'$ contains only one of $u,w$, say $u$. By symmetry, we may assume $\varphi(v_1) \neq \varphi(w)$. Then, it follows from Claim~\ref{claim_shifting_consecutive_vertices} that there is a coloring $\varphi^*$ of $D-v$ with $\varphi^*(u) = \varphi(v_1) \neq \varphi(w) = \varphi^*(w)$ and so we are done. 

Now let $\varphi^*(u)=k$ and $\varphi^*(w)=\ell$. We claim that the vertices of $C-v$ are colored alternately with $k$ and $\ell$. Otherwise, there are vertices $u_1,u_2$ distinct from $v$ with $u_1u_2 \in E(C)$ and $\{\varphi^*(u_1),\varphi^*(u_2)\} \neq \{k, \ell\}$. By symmetry, we may assume that $\ell \not \in \{\varphi^*(u_1),\varphi^*(u_2)\}$.
Then, by Claim~\ref{claim_shifting_consecutive_vertices}, there is a coloring $\varphi'$ of $D-v$ with $\varphi'(u) = \varphi^*(u_1)$ and $\varphi'(w) = \varphi^*(u_2)$. As $\ell \not \in \{\varphi^*(u_1),\varphi^*(u_2)\}$, and as $u$ and $w$ are the only neighbors of $v$ from $V(C)$ in $G(D)$, we obtain that either
$$d_{D[(\varphi')^{-1}(\{\ell \})] +v}^+(v) \neq d_{D[(\varphi^*)^{-1}(\{\ell \})] +v}^+(v) = f_\ell(v)$$
or
$$d_{D[(\varphi')^{-1}(\{\ell\})] +v}^-(v) \neq  d_{D[(\varphi^*)^{-1}(\{\ell \})] +v}^-(v)=f_\ell(v),$$ 
in contradiction to Claim~\ref{claim_degree_in_D_i}(b). This proves the claim that the vertices of $C-v$ are colored alternately with $k$ and $\ell$. As $\varphi^*(u) \neq \varphi^*(w)$, this implies that $C$ is an odd cycle and so the proof of the claim is complete.
\end{proof2}

As a consequence of the above claim, we obtain that every even cycle in $G(D)$ has at least two chords. Thus, it follows from Lemma~\ref{lemma_gallai-chords} that every block of $G(D)$ is an odd cycle or a complete graph. As $D$ and therefore $G(D)$ itself is a block, we conclude that $G(D)$ is either a cycle of odd length or a complete graph. To complete the proof of Theorem~\ref{theorem_main-result}, we show that both cases are impossible.

\begin{claim} \label{claim_odd-cycle}
$G(D)$ is not an odd cycle. 
\end{claim}
\begin{proof2}
For otherwise, let $v \in V(D)$ be an arbitrary vertex. As $C=G(D)$ is an odd cycle, $C$ obviously has no chords and so it follows from Claim~\ref{claim_no-chords} that there is a coloring $\varphi$ of $D-v$ and indices $k \neq \ell$ from $[1,p]$ so that the vertices of $C-v$ are colored alternately with $k$ and $\ell$. 
%
Then, it follows from Claim~\ref{claim_degree_in_D_i}(b) that 
\begin{align}
\begin{gathered}
f_k(v)=f_\ell(v)=1 \text{ and } f_i(v)=0 \text{ for } i\in [1,p] \setminus \{k,\ell\}\\
\text{and } \{uv,vu,vw,wv\} \subseteq A(D) \text{ where } u \text{ and } w \text{ are the neighbors of } v \text{ in } C .
\end{gathered}
\label{align_type_c}
\end{align}
By shifting the color from the neighbor $u$ of $v$ to $v$ we obtain \eqref{align_type_c} for $u$ instead of $v$. Repeating this argument proves that $(D,f)$ is a hard pair of type (C), a contradiction.
\end{proof2}

It remains to consider the case that $G(D)$ is a complete graph. {Claim \ref{claim_degree_in_D_i}  implies} that $|D| \geq 3$. First we claim that $D$ is bidirected. For otherwise, there are two vertices $u,v \in V(D)$ with $uv \in A(D)$ but $vu \not \in A(D)$. Let $\varphi$ be an $f$-coloring of $D-v$. Then, by Claim~\ref{claim_degree_in_D_i}(a), there is a vertex $w \in V(D) \setminus \{v\}$ with $\varphi(w) \neq \varphi(u)$. As $G(D[\{u,v,w\}])$ is a triangle, it follows from Claim~\ref{claim_shifting_consecutive_vertices}  that swapping the colors of $u$ and $w$ results in another $f$-coloring of $D-v$. Then, by Claim~\eqref{claim_degree_in_D_i}(b), we obtain that $wv \in A(D)$ but $vw \not \in A(D)$. By shifting the color from $w$ to $v$ we conclude from  Claim~\ref{claim_degree_in_D_i}(b) that $uw \in A(D)$ but $wu \not \in A(D)$. Similarly, by instead shifting the color from $u$ to $w$ we conclude from Claim~\eqref{claim_degree_in_D_i}(b)  that $wu \in A(D)$, but $uw \not \in A(D)$, which clearly is impossible. This proves the claim that $D$ is bidirected and so $D$ is a bidirected complete graph. 

Finally, we prove that $(D,f)$ is a hard pair of type (K), leading to the desired contradiction. To this end, let $v \in V(D)$ be an arbitrary vertex, let $(D_1,D_2,\ldots,D_p)$ be an $f$-partition of $D-v$, and let $n_i=|D_i|$. As $D$ is a bidirected complete graph, it then follows from Claim~\ref{claim_non-partitionable_eulerian}(a) that $f_i(v)=n_i$ for all $i \in [1,p]$. Moreover, if $u \in V(D) \setminus \{v\}$, say $u \in V(D_1)$ (by symmetry), then the sequence $(D_1',D_2',\ldots,D_p')$ with $D_1'=(D_1 + v) - u$ and $D_j'=D_j$ for $j \in [2,p]$ is an $f$-partition of $D-u$ (by Claim~\ref{claim_non-partitionable_eulerian}(c)) with $|D_i|'=|D_i|$ for all $i \in [1,p]$ and applying Claim~\ref{claim_non-partitionable_eulerian}(a) to $u$ leads to $f_i(u)=n_i$ for all $i \in [1,p]$. Note that $$n_1 + n_2 + \ldots + n_p = |D_1| + |D_2| + \ldots + |D_p| = |V(D)| - 1$$
and so $(D,f)$ is a hard pair of type (K), contradicting (3). This contradiction completes the proof of Theorem~\ref{theorem_main-result}.
\end{proof}

Since a bidirected graph $D$ is weakly $h$-degenerate if and only if its underlying graph $G(D)$ is strictly $h$-degenerate (\emph{i.e.} in each non-empty subgraph $G'$ of $G$ there is a vertex $v$ with $d_{G'}(v) < h(v)$), the restriction of Theorem~\ref{theorem_main-result} to bidirected graphs gives us also a result regarding strict degeneracy of undirected graphs. This result for undirected graphs was originally proven by Borodin, Kostochka, and Toft in the paper~\cite{BorKosToft}. The proof structure of Claim~\ref{claim_block} is similar to that of the first part of their proof. However, apart from this, our proof leads to another proof for the undirected case that uses completely different methods than the original proof. A major benefit of our proof is that it---contrary to the original proof in~\cite{BorKosToft}--- generalizes easily to the case of {directed multigraphs} respectively multigraphs. Of course, the definition of weak degeneracy and of $f$-partitions also works if we allow multiple arcs between vertices going in the same direction. Also, the definition of hard pair needs only to be generalized slightly. To this end, we need the following term. Let $G$ be a simple graph and let $t \geq 1$ be an integer. Then we denote by $tG$ the graph that results from $G$ by replacing each edge by $t$ parallel edges. Now let $\mathcal{D}$ be a connected multidigraph, and let $f:V(\mathcal{D}) \to \mathbb{N}_0^p$ be a vector function. We say that $(\mathcal{D},f)$ is a \textbf{hard pair}, if either
\begin{itemize}
\item $(\mathcal{D},f)$ is of type (M), or
\item $\mathcal{D}$ is a bidirected $tK_n$ and there are integers $n_1,n_2,\ldots,n_p \geq 0$ with at least two $n_i$ different from zero such that $n_1 + n_2 + \ldots + n_p=n-1$ and $$f(v)=(tn_1,tn_2,\ldots,tn_p) \text{ for all } v \in V(\mathcal{D}) \text{, or}$$
\item $\mathcal{D}$ is a bidirected $tC_n$ with $n$ odd and there are two indices $k \neq \ell$ from the set $[1,p]$ such that for all $v \in V(\mathcal{D})$,we have $f_i(v)=t$ if $i \in \{k,\ell\}$ and $f_i(v)=0$ otherwise, or
\item $(\mathcal{D},f)$ is obtained from two hard pairs via the merging operation.
\end{itemize}
For multidigraphs, we obtain the following theorem.
\begin{theorem}\label{theorem_main-result-multi}
Let $\mathcal{D}$ be a connected multidigraph, let $p \geq 1$, and let $f:V(\mathcal{D}) \to \mathbb{N}_0^p$ be a vector function such that $f_1(v) + f_2(v) + \ldots + f_p(v) \geq \max \{d_{\mathcal{D}}^+(v), d_\mathcal{D}^-(v)\}$ for all $v \in V(\mathcal{D})$. Then, $\mathcal{D}$ is not $f$-partitionable if and only if $(\mathcal{D},f)$ is a hard pair.
\end{theorem}
By inspecting their proofs, it is easy to check that Claims 1-7 also hold for multidigraphs. Only the remaining part of the proof needs to be changed slightly, but still can easily be done adapting the methods described there. Therefore, we abstain from giving an extra proof.

\section{Applications of Theorem~\ref{theorem_main-result}}
\subsection{Brooks' Theorem for list-colorings of digraphs}
As mentioned in the introduction, Theorem~\ref{theorem_main-result} implies Harutyunyan and Mohar's Theorem~\ref{theorem_harut-mohar-list}~\cite{HaMo11}. Let us recall the theorem for the reader's convenience.
\begin{restatedtheorem*}[Theorem~\ref{theorem_harut-mohar-list} {(Harutyunyan and Mohar, 2011)}]
Let $D$ be a connected digraph, and let $L$ be a list-assignment such that $|L(v)| \geq \max \{d_D^+(v), d_D^-(v)\}$ for all $v \in V(D)$. Suppose that $D$ is not $L$-colorable. Then, the following statements hold:
\begin{itemize}
\item[\upshape (a)] $D$ is Eulerian and $|L(v)|= \max \{d_D^+(v), d_D^-(v)\}$ for all $v \in V(D)$.
\item[\upshape (b)] If $B \in \mathscr{B}(D)$, then $B$ is a directed cycle of length $\geq 2$, or $B$ is a bidirected complete graph, or $B$ is a bidirected cycle of odd length $\geq 5$.
\item[\upshape (c)] For each $B \in \mathscr{B}(D)$ there is a set $\Gamma_B$ of $\Delta^+(B)$ colors such that for every $v \in V(D)$, the sets $\Gamma_B$ with $B \in \mathscr{B}_v(D)$ are pairwise disjoint and $L(v)= \bigcup_{B \in \mathscr{B}_v(D)} \Gamma_B$. \hfill \ensuremath{_\diamond}
\end{itemize}
\end{restatedtheorem*}

\begin{proofof}[Theorem~\ref{theorem_harut-mohar-list}]
Let $D$ and $L$ be as described in the theorem. By using the method from the introduction, we transform the list-coloring problem to that of finding an $f$-partition: Let $\Gamma= \bigcup_{v \in V(D)}L(v)$. By renaming the colors if necessary, we may assume $\Gamma=[1,p]$. Now let $f:V(D) \to \mathbb{N}_0^p$ be the vector function with
$$f_i(v)=
\begin{cases}
1 \text{ if } i \in L(v), \text{ and}\\
0 \text{ if } i \not \in L(v)
\end{cases}
$$
for $i \in [1,p]$. By the definition of $f$, we have $f_1(v) + f_2(v) + \ldots + f_p(v) = |L(v)| \geq \max\{d_D^+(v),d_D^-(v)\}$ for all $v \in V(D)$ and $D$ is not $f$-partitionable. Thus, it follows from Theorem~\ref{theorem_main-result} that $(D,f)$ is a hard pair. Then, statement (a) follows from Proposition~\ref{prop_digraph_f-hard}(a). Moreover, Proposition~\ref{prop_digraph_block-hardpair}(a) implies that for each block $B \in \mathscr{B}$, $(B,f_B)$ is of type (M), (K), or (C), where $f_B$ is the function as defined in Proposition~\ref{prop_digraph_block-hardpair}. Note that if $(B,f_B)$ is of type (M), then from the definition of $f$ it follows that the only non-zero coordinate of $f_B$ is constant $1$ and, hence, $d_B^+(v)=d_B^-(v)=1$ for all $v \in V(B)$. Consequently, $B$ is a directed cycle. Thus, (b) holds true. Statement (c) follows easily from Proposition~\ref{prop_digraph_block-hardpair}(a) and (b).
\end{proofof}

\subsection{The $s$-degenerate dichromatic number}
{In ~\cite{Gol16}}, Golowich introduces a generalization of the dichromatic number as follows. Let $s \geq 1$ be an integer. Then, the \textbf{$s$-degenerate dichromatic number} of a digraph $D$, denoted by $\overr\chi_s(D)$, is the least integer $p$ such that $D$ admits a $p$-partition into weakly $s$-degenerate subdigraphs. Clearly, $\chi_1$ corresponds to the dichromatic number. Note that the $s$-degenerate dichromatic number is the digraph counterpart to the \textbf{point partition number} of an undirected graph, which was introduced by Lick and White~\cite{LiWhi72} and which is also known as the \textbf{$s$-chromatic number}. The point partition number $\chi_s$ of an undirected graph $G$ is usually defined as the minimum number $p$ such that $G$ admits a $p$-partition into $s$-degenerate subgraphs. By setting $f_i \equiv s$ in Theorem~\ref{theorem_main-result}, we easily obtain the following theorem, which was proven for undirected graphs by Mitchem~\cite{Mitch77}.

\begin{theorem} \label{theorem_s-degenerate}
Let $D$ be a connected digraph and let $m=\max_{v \in V(D)} \{d_D^+(v),d_D^-(v)\}$. Moreover, let $s \geq 1$ be an integer and let $p= \lceil \frac{m}{s} \rceil$. If $D$ is neither a bidirected odd cycle, a bidirected complete graph, nor an Eulerian digraph in which every vertex has in- and out-degree $s$, then $\chi_s(D) \leq p$.
\end{theorem}

To conclude the paper, let us demonstrate how to obtain a list-version of Theorem~\ref{theorem_s-degenerate}. Given a digraph $D$ and a list-assignment $L$ of $D$, we say that $D$ is $(L,s)$\textbf{-colorable} if $D$ admits an $L$-coloring in which every color class induces a weakly $s$-degenerate subdigraph.

\begin{theorem}
Let $D$ be a connected digraph and let $m=\max_{v \in V(D)} \{d_D^+(v),d_D^-(v)\}$. Moreover, let $s \geq 1$ be an integer and let $L$ be a list-assignment with  $|L(v)| \geq  \frac{m}{s}$ for all $v \in V(H)$. Then, $D$ is not $(L,s)$-colorable if and only if the following two conditions are fulfilled:
\begin{itemize}
\item[\upshape (a)] $D$ is a bidirected complete graph with $|D|-1 \equiv 0 \text{ (mod } s \text{)}$, or $D$ is a bidirected odd cycle and $s=1$, or $D$ is an Eulerian digraph in which every vertex has in-degree and out-degree $s$.
\item[\upshape (b)] There is a color set $\Gamma$ such that $L(v)=\Gamma$ for all $v \in V(D)$ and $|\Gamma|= m/s$.
\end{itemize}
\end{theorem}

\begin{proof}
If $m=0$, then $D$ consists of just one vertex and the statement is evident. So assume $m \geq 1$. Let $\Gamma=\bigcup_{v \in V(D)} L(v)$. By renaming the colors if necessary, we can assume that $\Gamma=[1,p]$. Now we define a vector function $f:V(D) \to \mathbb{N}_0^p$ as follows. Let 
$$f_i(v)=\begin{cases}
s \quad \text{if } i \in L(v),\\
0 \quad \text{otherwise}
\end{cases}$$
for $i \in [1,p]$. Then, $D$ is not $(L,s)$-colorable if and only if $D$ is not $f$-partitionable.
By definition, we have $f_1(v) + f_2(v) + \ldots + f_p(v) = s |L(v)| \geq m$ for all $v \in V(D)$ and so $D$ is not $(L,s)$-colorable if and only if $(D,f)$ is a hard pair. (by Theorem~\ref{theorem_main-result}) 

First assume that $D$ and $L$ satisfy (a) and (b). Then it is easy to check that $(D,f)$ is indeed a hard pair and so $D$ is not $(L,s)$-colorable. Now assume that $D$ is not $(L,s)$-colorable. Then $(D,f)$ is a hard pair and from Proposition~\ref{prop_digraph_f-hard}(a) it follows that $m \leq s |L(v)| = d_D^+(v)=d_D^-(v)$ for all $v \in V(D)$ and so $D$ is Eulerian and each vertex $v$ satisfies $d_D^+(v)=d_D^-(v)=m$. If $m=s$, then $|L(v)|=1$ for all $v \in V(D)$ and $D$ is not $(L,s)$-colorable if and only if there is a color $\alpha \in \Gamma$ with $L(v)=\{\alpha\}$ for all $v \in V(D)$, and we are done. So assume $m > s$. Then, $|L(v)| \geq 2$ and so for each vertex $v$ there are two indices $i \neq j$ such that $f_i(v) \geq 1$ and $f_j(v) \geq 1$. Consequently, by Proposition~\ref{prop_digraph_block-hardpair}, no end-block of $D$ is a mono-block. Since every vertex of $D$ has in-degree and out-degree $m$, this implies that $D$ itself is a block and so $D$ is either a bidirected complete graph or a bidirected odd cycle. In the first case, $(D,f)$ is of type (K) and, since every coordinate of $f$ is either $s$ or zero we easily conclude from the definition of hard pair of type (K) that $|D| - 1\equiv 0 \text{ (mod } s \text{)}$. If $(D,f)$ is of type (C), we conclude that $s=1$. Moreover, in both cases the function $f$ is constant and so the list-assignment $L$ is constant, too, and $|L(v)|=m/s$ for all $v \in V(D)$. This completes the proof.
\end{proof}

\section{{Appendix: Constructing  an $f$-partition}}
{The problem of deciding whether}  the dichromatic number of a digraph is at most two is already \NP-complete~\cite{Bokal04}. Since this problem {is the same as deciding whether} there is an $f$-partition for $f\equiv (1,1)$, we cannot hope {to find an efficient  algorithm for deciding, for a given pair $(D,f)$, whether $D$ admits an $f$-partition}.  The good news, however, is that it is possible to deduce a polynomial time algorithm from our proof that, given a pair $(D,f)$ with $f_1(v) + f_2(v) + \ldots + f_p(v) \geq \max\{d_D^+(v),d_D^-(v)\}$ for all $v \in V(D)$, either verifies that $(D,f)$ is a hard pair or returns an $f$-partition of $D$. The algorithm uses the following six subroutines. 

\begin{algorithm}[H]
\caption{Greedy Coloring}\label{algorithm_greedy}
\textbf{Input:} $(D,f,v^*)$ {such that
   $D$ is connected,
  $f_1(v) + f_2(v) + \ldots + f_p(v) \geq \max\{d_D^+(v), d_D^-(v)\}$ for all $v \in V(D)$
  and $f_1(v^*) + f_2(v^*) + \ldots + f_p(v^*) > \min\{ d_D^+(v^*), d_D^-(v^*)\}$.
  }\\
\textbf{Output:} {An} $f$-partition $(D_1,D_2,\ldots,D_p)$ of $D$
\begin{algorithmic}[1]
\State $n := |D|$, $v_n := v^*$
\For{$i=n-1,n-2,\ldots,1$}
\State let $v_i$ be a vertex such that $v_i$ has a neighbor from $\{v_{i+1},v_{i+2}, \ldots, v_n\}$
\EndFor
\State $D_j : =\varnothing$ for $j \in [1,p]$
\For{$i=1,2,\ldots,n$}
\State $D_j:=D_j \cup \{v_i\}$ where $j$ is the minimum integer with \label{step_find_j}\\ 
\hspace*{\algorithmicindent}$f_j(v_i) > \min \{d_{D_j + v_i}^+(v_i), d_{D_j + v_i}^-(v_i)\}$.\\
\EndFor

\State \Return $(D_1,D_2,\ldots,D_p)$.
\end{algorithmic}
\end{algorithm}

\begin{algorithm}[H]
\caption{Greedy Coloring of $D-v$}\label{algorithm_greedy_D-v}
\textbf{Input:} $(D,f,v)$ {where $D$ is connected and} $f_1(v) + f_2(v) + \ldots + f_p(v) = \max\{d_D^+(v), d_D^-(v)\}$ for all $v \in V(D)$\\
\textbf{Output:} {An} $f$-partition $(D_1,D_2,\ldots,D_p)$ of $D-v$
\begin{algorithmic}[1]
\State {Let $K_1,\ldots{},K_q$  be the connected components of $D-v$.} 
\For{{$j=1,2,\ldots{},q$}}
\State {find a vertex $v_{K_j}^*$ with $f_1(v_{K_j}^*) + f_2(v_{K_j}^*) + \ldots + f_p(v_{K_j}^*) > \min\{ d_K^+(v_{K_j}^*), d_K^-(v_{K_j}^*)\}$}
\State {apply Algorithm~\ref{algorithm_greedy} to $(K,f,v_{K_j}^*)$ to obtain an $f$-partition $(D_{{K_j},1},D_{{K_j},2},\ldots,D_{{K_j},p})$ of $K_j$}.
\EndFor
\For{$i=1,2,\ldots,{p}$}
\State {$D_i := \bigcup_{j\in [1,q]}D_{K_j,i}$}
\EndFor
\State \Return $(D_1,D_2,\ldots,D_p)$.
\end{algorithmic}
\end{algorithm}

 The correctness of the first two algorithms follows from Claims~\ref{claim_D-v} and~\ref{claim_non-partitionable_eulerian}. In particular, in step~\ref{step_find_j} of Algorithm~\ref{algorithm_greedy}, we always find such an index $j$ since 

\begin{align*}
\sum_{k \in [1,p]} \min\{d_{D_k + v_i}^+(v_i),d_{D_k+v_i}^-(v_i)\} & \leq \min\{d_{D[\{v_1,v_2,\ldots,v_i\}]}^+(v_i), d_{D[\{v_1,v_2,\ldots,v_i\}]}^-(v_i)\}\\
& < \max \{d_D^+(v_i), d_D^-(v_i)\} \leq f_1(v_i) + f_2(v_i) + \ldots + f_p(v_i).
\end{align*}

{\noindent{}Here the last in-equality follows from the fact that $v_i$ has a neighbor in the set $\{v_{i+1},\ldots{},v_n\}$.}

Algorithm~\ref{algorithm_shifting} describes the shifting procedure that we introduced after Claim~\ref{claim_two-colors}. Given a vertex $v$ and {a neighbor $w$ of $v$}, we first check if we can add $v$ to the $f$-partition in order to obtain the desired $f$-partition of $D$. If this is not the case, {then it follows from Claim
  \ref{claim_non-partitionable_eulerian}(c) that }we can ''swap`` $v$ and {$w$} and obtain an $f$-partition of $D-w$. The idea of the Main Algorithm~\ref{main_algorithm} is to iteratively split end-blocks $B$ together with the appropriate block-function $f_B$ (as in Proposition~\ref{prop_digraph_block-hardpair}) from $D$ and try to partition those separately. If we find an $f_B$-partition of some block $B$, we describe a procedure how to extend this partition to an $f$-partition of the whole digraph.

\begin{algorithm}
\caption{Shifting Procedure}\label{algorithm_shifting}
\textbf{Input:} $(D,f,v,{w},(D_1,D_2,\ldots,D_p))$ where
\begin{algorithmic}[1]
\Statex \textbullet~$D$ is a {E}ulerian block
\Statex \textbullet~{$v$ is a vertex of $D$}
\Statex \textbullet~{$w$ is a neighbor of $v$ in $D$}
\Statex \textbullet~$f_1({u}) + f_2({u}) + \ldots + f_p({u})= d_D^+({u}) = d_D^-({u})$ for all ${u} \in V(D)$
\Statex \textbullet~$(D_1,D_2,\ldots,D_p)$ $f$-partition of $D-v$
\end{algorithmic}
\textbf{Output:} Either {an} $f$-partition  of $D$, or
{an} $f$-partition of $D-w$.\\
\hrule
\begin{algorithmic}[1]
\For{$i=1,2,\ldots,p$}
\If{$\min\{d_{D_i+v}^+(v),d_{D_i+v}^-(v)\} < f_i(v)$} \label{line_shifting_degree}
\State $D_i := D_i + v$
\State \Return {the $f$-partition $(D_1,D_2,\ldots,D_p)$} of $D$ 
\EndIf
\EndFor
\State let $i$ be the index with $w \in V(D_i)$
\State $D_i:=D_i + v - w$
\State \Return {the $f$-partition $(D_1,D_2,\ldots,D_p)$} of $D-w$.
\end{algorithmic}
\end{algorithm}

Algorithm~\ref{algorithm_cycle} deals with the case that the underlying graph {$C=G(B)$} of a block $B$ is a cycle and $(B,f_B)$ is not a hard pair. Then, it follows from the proof of Claim~\ref{claim_odd-cycle} that by repeated application of the shifting procedure we will eventually obtain an $f_B$-partition. {As $(B,f_B)$ is not a hard pair, there are two consecutive vertices $v,w$ on $C$ with $f_B(v) \neq f_B(w)$. By beginning the shifting with the vertex $v$, we ensure that} the algorithm will terminate after at most {$\text{length}(C)+1$} iterations {since then the condition in Algorithm~\ref{algorithm_shifting} line~\ref{line_shifting_degree} will be violated at the latest. This is due to the fact that if $u,u'$ are the neighbors of $v$, respectively $w$ on $C$ distinct from $\{v,w\}$ and $(D_1,D_2,\ldots,D_p)$ is an $f_B$-partition of $D-v$, then $u$ and $u'$ need to belong to different parts of the partition (by Claim~\ref{claim_two-colors}(b) applied to $(D_1,D_2,\ldots,D_p)$ and to the $f$-partition of $D-w$ that results from $(D_1,D_2,\ldots,D_p)$ by swapping $v$ and $w$). Thus, after $\text{length}(C)+1$ iterations, $v$ is again uncolored but the neighbors of $v$ will belong to other partition parts than previous. Hence, Algorithm~\ref{algorithm_cycle} runs in polynomial time.}

\begin{algorithm}[H]
\caption{$f$-partition of a cycle}\label{algorithm_cycle}
\begin{algorithmic}[1]
\Procedure{CyclePartition}{$D,f$}
\State let $v,w$ be {two adjacent vertices of $D$} with $f(v) \neq f(w)$
\State apply Algorithm~\ref{algorithm_greedy_D-v} to $(D,f,v)$
\While{true}
\State apply Algorithm~\ref{algorithm_shifting} to $(D,f,{v,w},(D_1,D_2,\ldots,D_p))$
\If{Algorithm~\ref{algorithm_shifting} returns {an} $f$-partition $(D_1',D_2',\ldots,D_p')$ of $D$}
\State \Return {the $f$-partition} $(D_1',D_2',\ldots,D_p')$ of $D$
\Else  
\State Algorithm~\ref{algorithm_shifting} returns {an} $f$-partition $(D_1',D_2',\ldots,D_p')$ of $D-w$
\State {$u:=$ neighbor of $w$ in $D$ distinct from $v$}, $v:=w$, {$w:=u$}, 
\Statex \hspace*{\algorithmicindent}\hspace*{\algorithmicindent}\hspace*{\algorithmicindent}$(D_1,D_2,\ldots,D_p):=(D_1',D_2',\ldots,D_p')$
\EndIf 
\EndWhile\label{cycle_while}
\EndProcedure
\end{algorithmic}
\end{algorithm}

Algorithm~\ref{algorithm_complete} describes a procedure that settles the case that the underlying graph of a block $B$ is a complete graph. Here, we first need to check if $B$ is bidirected; otherwise, we describe a method how to obtain an $f_B$-partition of the non-bidirected block (lines~$1$ till \ref{step_not-bidirected}). That the procedure works follows from the proof's part after Claim~\ref{claim_odd-cycle}. If $B$ is bidirected, we argue as in the same proof section to get an $f$-partition.

\begin{algorithm}[H]
\caption{$f$-partition of a complete digraph}\label{algorithm_complete}
\begin{algorithmic}[1]
\Procedure{CompleteDigraphPartition}{$D,f$}
\If{there is an arc $vu \in A(D)$ with $uv \not \in A(D)$}
	\State apply Algorithm~\ref{algorithm_greedy_D-v} to $(D,f,v)$ to obtain $f$-partition $(D_1,D_2,\ldots,D_p)$ of $D-v$.
	\If{$\min \{d_{D_j+v}^+(v),d_{D_j+v}^-(v)\} < f_j(v_i)$ for some $j \in [1,p]$}
	\State $D_j:=D_j + v_i$
	\State \Return $f$-partition $(D_1,D_2,\ldots,D_p)$.
	\ElsIf{all but one $D_i$ are empty}
			\State find vertex $w \in V(D)$ with $f_j(w) > 0$ for some $j \neq i$ (if possible, $w=v$).
			\State $D_i:=(D_i+v)-w$, $D_j:=D_j+w$
			\State \Return $(D_1,D_2,\ldots,D_p)$ $f$-partition of $D$
	\Else 
			\State let $i$ be the index with $u \in V(D_i)$. Find $w$ with $w \in V(D_j)$ and $i \neq j$.
			\State $v^*:=v$, $C=G(D)[\{u,v,w\}]$
			\While{true}
			\State apply Algorithm~\ref{algorithm_shifting} to $(D,f,C,v^*,(D_1,D_2,\ldots,D_p)$
			\If{Algorithm~\ref{algorithm_shifting} returns $f$-partition $(D_1',D_2',\ldots,D_p')$ of $D$}
				\State $(D_1,D_2,\ldots,D_p):=(D_1',D_2',\ldots,D_p')$
				\State \Return $(D_1,D_2,\ldots,D_p)$ is $f$-partition of $D$
			\Else  
				\State Algorithm~\ref{algorithm_shifting} returns $f$-partition $(D_1',D_2',\ldots,D_p')$ of $D-w^*$
				\Statex\hspace*{\algorithmicindent}\hspace*{\algorithmicindent}\hspace*{\algorithmicindent}\hspace*{\algorithmicindent}\hspace*{\algorithmicindent}where $w^*$ is the right neighbor of $v^*$ on $C$
				\State $v^*:=w^*$, $(D_1,D_2,\ldots,D_p):=(D_1',D_2',\ldots,D_p')$
			\EndIf 
			\EndWhile
		\EndIf \label{step_not-bidirected}
	\Else ~$D$ is a bidirected complete graph. Let $v \in V(D)$ be arbitrary and apply
\Statex\hspace*{\algorithmicindent}\hspace*{\algorithmicindent}Algorithm~\ref{algorithm_greedy_D-v} to obtain an $f$-partition $(D_1,D_2,\ldots,D_p)$ of $D-v$.
	\State let $n_i=|D_i|$ for $i=1,\ldots,p$.
	\If{$f_i(v) > n_i$ for some $i \in [1,p]$}
		\State $D_i:=D_i + v$
		\State \Return $(D_1,D_2,\ldots,D_p)$ is an $f$-partition of $D$
	\ElsIf{there is a vertex $w$ with $w \in V(D_i)$ $(i \in [1,p])$ and $f_i(w) \geq n_i + 2$}
	 	\State $D_i:=D_i + v$
	 	\State \Return $(D_1,D_2,\ldots,D_p)$ is an $f$-partition of $D$
	\Else
		\State find a vertex $w$ and indices $i \neq j$ with $w \in V(D_i)$ and $f_j(w) > n_j$
		\State $D_i:=(D_i + v) - w$, $D_j:=D_j + w$
		\State \Return $(D_1,D_2,\ldots,D_p)$ is an $f$-partition of $D$
	\EndIf 
	\EndIf
\EndProcedure
\end{algorithmic}
\end{algorithm}

\newpage
Algorithm~\ref{algorithm_block} uses the two previously described procedures in order to find an $f_B$-partition of any block $B$ where $(B,f_B)$ is not a hard pair. We first check if $(B,f_B)$ is a hard pair (which can be done efficiently). If this is not the case, we determine if the underlying graph is a cycle or a complete graph and apply the respective procedures to get an $f_B$-partition. If the block's underlying graph is neither a complete graph nor a cycle, we argue as in Claim~\ref{claim_no-chords}. Note that Cranston and Rabern describe in~\cite{CraRa15} a constructive {proof of} Gallai's Lemma~\ref{lemma_gallai-chords}. In particular, they provide an efficient method how to find an even cycle that has at most one chord if the block is neither a complete graph nor an odd cycle. Moreover, in the worst case we need to shift a lot here but we still stay polynomial in the {size of the block}.

\begin{algorithm} 
\caption{Finding an $f$-partition of a block}\label{algorithm_block}
\textbf{Input:} $(D,f)$ where
\begin{algorithmic}[1]
\Statex \textbullet~$D$ is a Eulerian block, and
\Statex \textbullet~$f_1(v) + f_2(v) + \ldots + f_p(v)= d_D^+(v) = d_D^-(v)$ for all $v \in V(D)$
\end{algorithmic}
\textbf{Output:} Either {that} $(D,f)$ is a hard pair, or {an} $f$-partition $(D_1,D_2,\ldots,D_p)$ of $D$
\\
\hrule
\begin{algorithmic}[1]
\If{$(D,f)$ is a hard pair}
\State \Return $(D,f)$ is a hard pair
\ElsIf{$G(D)$ is a cycle} 
	\State $(D_1,D_2,\ldots,D_p)=$\textsc{CyclePartition($D,f$)}
	\State \Return $(D_1,D_2,\ldots,D_p)$ is an $f$-partition of $D$
\ElsIf{$G(D)$ is a complete graph}
	\State $(D_1,D_2,\ldots,D_p)=$\textsc{CompleteDigraphPartition($D,f)$}
	\State \Return $(D_1,D_2,\ldots,D_p)$ is an $f$-partition of $D$
\Else
	\State find {an} even cycle $C$ of $D(G)$ and vertex $v \in V(C)$ that is not contained in a chord
	\State apply Algorithm~\ref{algorithm_greedy_D-v} to obtain an $f$-partition $(D_1,D_2,\ldots,D_p)$ of $D-v$
	\If{$\min \{d_{D_j+v}^+(v),d_{D_j+v}^-(v)\} < f_j(v_i)$ for some $j \in [1,p]$}
	\State $D_j:=D_j + v_i$
	\State \Return $f$-partition $(D_1,D_2,\ldots,D_p)$.
	\EndIf
	\If{all but one $D_i$ are empty}
		\State find {a} vertex $w$ with $f_j(w)>0$ for some $j \neq i$
		\State $D_i:=D_i - w$, $D_j:=(\{w\},\varnothing)$.
	\EndIf
	\State let $u,w$ be the neighbors of $v$ on $C$
	\If{$u,w$ are contained in the same $D_i$}
		\State find vertex $z$ with $z \in V(D_j)$ and $j \neq i$
		\State find cycle $C'$ in $G(D)$ containing $z$ and the edge $vw$
		\Statex\hspace*{\algorithmicindent}\hspace*{\algorithmicindent}and choose a cyclic ordering of the vertices of $C'$
	\EndIf
	\State $v^*:=v$
	\While{$u,w$ are contained in the same $D_i$ or $v \neq v^*$}
		\State apply Algorithm~\ref{algorithm_shifting} to $(D,f,v^*,C',(D_1,\ldots,D_p))$
		\If{Algorithm~\ref{algorithm_shifting} returns $f$-partition $(D_1,D_2,\ldots,D_p)$ of $D$}
		\State \Return $(D_1,D_2,\ldots,D_p)$ is an $f$-partition of $D$
		\Else~$v^*:=$ right neighbor of $v^*$ on $C'$
		\EndIf 
	\EndWhile
	\State choose a cyclic ordering of the vertices of $C$
	\While{true}
	\State apply Algorithm~\ref{algorithm_shifting} to $(D,f,v^*,C,(D_1,\ldots,D_p))$
	\If{Algorithm~\ref{algorithm_shifting} returns $f$-partition $(D_1,D_2,\ldots,D_p)$ of $D$}
		\State \Return $(D_1,D_2,\ldots,D_p)$ is an $f$-partition of $D$
		\Else~$v^*:=$ right neighbor of $v^*$ on $C$
	\EndIf
	\EndWhile
\EndIf
\end{algorithmic}
\end{algorithm}

{Algorithm~\ref{main_algorithm}, which is the main algorithm, first} tries to color the vertices greedily by using Algorithm~\ref{algorithm_greedy}. If this is not possible, it immediately follows that $f_1(v)+f_2(v)+\ldots+f_p(v) = d_D^+(v) = d_D^-(v)$ for all $v \in V(D)$ and, hence, $D$ is Eulerian. Note that this implies that every block is Eulerian, too (see Claim~\ref{claim_block}). We then compute a block-decomposition of $D$ (which can be done efficiently, see \cite{hopcroftACM16}). Afterwards, we take an end-block $B_1$ of $D$ as well as the {unique} separating vertex $v_1$ of $D$ contained in $B_1$ and color the vertices of $D-v_1$ greedily {using} Algorithm~\ref{algorithm_greedy_D-v}. If we cannot add $v_1$ to any of the partition parts, then it follows from Claim~\ref{claim_block} that we can split $f(v_1)$ into $f_{B_1}(v_1)$ and an updated $f(v_1):=f(v_1) - f_{B_1}(v_1)$ as in lines~\ref{step_block-function}-\ref{step_block-function_end}, so that both $(B_1,f_{B_1})$ as well as the remaining part $D \setminus (V(B_1) \setminus\{v_1\})$ together with the new function $f$ fulfill the degree-condition. Then we try to find an $f_{B_1}$-partition of $B_1$, if this is not possible we {record} that $(B_1,f_{B_1})$ is a hard pair. Afterwards, we continue splitting end-blocks from $D$ until we can either confirm that every block together with its corresponding function is a hard pair or we find an $f_{B_i}$-partition of some block $B_i$. Then, it follows from {the proof of} Claim~\ref{claim_block} that we can extend this partition to an $f$-partition of the current $D$. Now we can greedily color the vertices of $B_{i-1} \setminus \{v_{i-1}\}$ and combine this partition with the partition of the previous $D$ in order to get an $f$-partition of $D + B_{i-1}$. By repeating this procedure until $B_1$ gets added to $D$, we eventually obtain the desired $f$-partition.

\begin{algorithm} 
\caption{Main Algorithm}\label{main_algorithm}
\textbf{Input:} $(D,f)$ where
\begin{algorithmic}[1]
\Statex \textbullet~$D$ is a connected digraph
\Statex \textbullet~$f_1(v) + f_2(v) + \ldots + f_p(v) \geq \max\{d_D^+(v),d_D^-(v)\}$ for all $v \in V(D)$
\end{algorithmic}
\textbf{Output:} Either $(D,f)$ is a hard pair, or $f$-partition $(D_1,D_2,\ldots,D_p)$ of $D$
\\
\hrule
\begin{algorithmic}[1]
\If{there is a vertex $v^*$ with $f_1(v^*) + f_2(v^*) + \ldots + f_p(v^*) > \min \{d_D^+(v^*),d_D^-(v^*)\}$}
 	\State apply Algorithm~\ref{algorithm_greedy} to $(D,f,v^*)$ to obtain an $f$-partition $(D_1,D_2,\ldots,D_p)$ of $D$
 	\State \Return $(D_1,D_2,\ldots,D_p)$ is an $f$-partition of $D$
\EndIf
\State Compute {the} block-decomposition of $D$
\State $i:=1$
\While{$D \neq \varnothing$}
\State  let $B_i$ be an end-block of $D$, let $v_i$ be the separating vertex contained in $B_i$.
\State apply Algorithm~\ref{algorithm_greedy_D-v} to $(D,v_i)$ to obtain an $f$-partition of $D-v_i$.
\If{$\min \{d_{D_j+v_i}^+(v_i),d_{D_j+v_i}^-(v_i)\} < f_j(v_i)$ for some $j \in [1,p]$}
	\State $D_j:=D_j + v_i$
	\State go to step~\ref{while_recursive_coloring}
\Else 
	\State $f_{B_i}(v_i):=(f_{B_i,1}(v_i),f_{B_i,2}(v_i),\ldots,f_{B_i,p}(v_i))$ with \label{step_block-function}
	\State $f_{B_i,j}(v_i)=d_{B_i \cap (D_j + v_i)}^+(v_i)$ for $j \in [1,p]$
	\State $f_{B_i}(w)=f(w)$ for $w \in V(B_i) \setminus \{v_i\}$. 
	\State apply Algorithm~\ref{algorithm_block} to $(B_i,f_{B_i})$
	\If{Algorithm~\ref{algorithm_block} returns $f_{B_i}$-partition of $B_i$}
		\State $D_i:=D[(V(D_i)\cap(V(D) \setminus V(B))) \cup V(D_i')]$ for $i \in[1,p]$
		\State go to step~\ref{while_recursive_coloring}
	\Else
	\State $D:=D-(V(B)\setminus\{v_i\})$, $f(v_i):= f(v_i) - f_{B_i}(v_i)$, update block-decomposition of $D$, $i:=i+1$ \label{step_block-function_end}
	\EndIf
\EndIf
\EndWhile
\State \Return $D$ is $f$-hard and results from merging $(B_1,f_{B_1})$,$(B_2,f_{B_2}), \ldots, (B_{i-1},f_{B_{i-1}})$
\While{$i > 1$} \label{while_recursive_coloring}
	\State apply Algorithm~\ref{algorithm_greedy_D-v} on $(B_{i-1}, f_{B_{i-1}}, v_{i-1})$, get an $f_{B_{i-1}}$-partition $(D_1',D_2',\ldots,D_p')$
	\Statex\hspace*{\algorithmicindent} of $B_{i-1}-v_{i-1}$
	\State $D_i:=D[V(D_i) \cup V(D_i')]$ for $i \in [1,p]$
	\State $i := i-1$
\EndWhile
\State \Return $(D_1,D_2,\ldots,D_p)$ is an $f$-partition of $D$. 
\end{algorithmic}
\end{algorithm}
\end{document}